\documentclass[11pt]{article}
\usepackage[a4paper,right = 22mm, left=22mm, top=25mm, bottom = 25mm]{geometry}

\usepackage{lipsum}
\usepackage{amsfonts}
\usepackage{amsthm}
\usepackage{amsmath}
\usepackage{graphicx}
\usepackage{epstopdf}
\usepackage{algorithmic}
\usepackage{multicol}
\usepackage{nicefrac}
\usepackage{authblk}
\usepackage{algorithm}
\usepackage{xcolor}
\usepackage{algorithmic}
\input{abbrev}

\newcommand{\true}{_{\text{\scriptsize{true}}}}

\newcommand{\calEk}{{\calE}^{(k)}_{k+1}}

\newcommand{\Rki}{\bfR_{k+1}^{-1}}
\newcommand{\whAEk}{\widehat{\bfA}_{k+1}^{\calE}}
\newcommand{\AEkT}{(\bfA_{k+1}^{\calE})\t}
\newcommand{\AcalE}{(\bfA^{\calE})}
\newcommand{\bRk}{\bar{\bfR}_{k+1}}
\newcommand{\bRSk}{\bar{\bfR}_{k+1}^{(S)}}
\DeclareMathOperator*{\argmin}{argmin}  

\newtheorem{remark}{Remark}

\newtheorem{proposition}{Proposition}

\newcommand{\msl}[1]{\textcolor{black}{#1}}

\numberwithin{equation}{section}
\title{New flexible  and inexact Golub-Kahan algorithms\\for inverse problems}
\author[1]{Silvia Gazzola}
\author[2]{Malena Sabaté Landman}
\affil[1]{Dipartimento di Matematica, Università di Pisa, \textit{silvia.gazzola@unipi.it}.}
\affil[2]{Mathematical Institute, Oxford, \textit{malena.sabatelandman@maths.ox.ac.uk}}

\date{}
\begin{document}

\maketitle
\begin{abstract}
This paper introduces a new class of algorithms for solving large-scale linear inverse problems based on new flexible and inexact Golub-Kahan factorizations. The proposed methods iteratively compute regularized solutions by approximating a solution to (re)weighted least squares problems via projection onto adaptively generated subspaces, where the constraint subspaces for the residuals are (formally) equipped with iteration-dependent preconditioners or inexactness. The new solvers offer a flexible and inexact Krylov subspace alternative to other existing Krylov-based approaches for handling general data fidelity functionals, e.g., those expressed in the  $p$-norm. Numerical experiments in imaging applications, such as 
image deblurring and computed tomography, 
highlight the effectiveness and competitiveness of the proposed methods with respect to other popular methods. \\

\normalsize
\end{abstract}

\section{Introduction}\label{sec: Intro}
Discrete linear inverse problems are formulated as the solution of 
large-scale linear systems of equations of the form
\begin{equation}\label{eq: linsys}
\bfA\bfx\true+\bfn=\bfb\,,
\end{equation}
where the discretized forward operator $\bfA\in \bbR^{m\times n}$ is large-scale with ill-determined rank, $\bfx\true\in\bbR^n$ is the unknown quantity of interest, and $\bfn\in \bbR^m$ are some unknown perturbations (noise) affecting the available data $\bfb\in \bbR^m$. Such problems arise in many relevant applications in Science and Engineering, including image deblurring and computed tomography. Due to the ill-posedness of $\bfA$ and the presence of noise in $\bfb$, one should regularize \eqref{eq: linsys} in order to recover a meaningful approximation of $\bfx\true\in \bbR^n$, i.e., replace the ill-posed linear system by a nearby problem, whose solutions is more stable with respect to perturbations in the data. We refer to \cite{HaHa93, hansen2010discrete, HaNaOL06} for more details on discrete inverse problems and classical regularization methods.

Many iterative solvers that look for an approximation 
\begin{equation}\label{eq: LS}
\bar\bfx^\ast=\argmin_{\bfx\in\bbR^n}\nicefrac{1}{2}\|\bfA\bfx-\bfb\|_2^2
\end{equation}
of $\bfx\true$ in \eqref{eq: linsys}, 
including many Krylov subspace methods, are known to have regularizing properties achieved by early termination. Namely, the first iterations are unaffected by noise and the approximate solution approaches the true solution, but this stops once the noisy components and the unwanted solution of the unregularized least squares (LS) problem \eqref{eq: LS} start to be recovered. This type of behavior is usually referred to as semiconvergence; see \cite{pchDPC, hansen2010discrete}. Among Krylov methods, the mathematically equivalent CGLS and LSQR can handle rectangular matrices in \eqref{eq: LS} and their regularizing properties are well understood; see \cite{HANSEN2025} and \cite[Chapter 6]{Han97}. 
Essentially, when applying such methods to \eqref{eq: LS}, we are trying to `fit as much of the data, fitting the minimum amount of noise'. Even if such solvers can be efficiently employed with additional (often Tikhonov-like) regularization with the goal of mitigating semiconvergence, for simplicity this is not considered in this paper; see \cite{Chung2024review} for more details.

In this paper we consider iterative regularization methods that solve generalizations of problem \eqref{eq: LS}, whereby some statistical information about the noise $\bfn$ affecting the data in \eqref{eq: linsys} is encoded into the problem formulation to enable more accurate recovery of $\bfx\true$. 
More specifically, if $\bfR\in\bbR^{m \times m}$ is the noise covariance matrix (symmetric positive definite (SPD) by construction), the well-known Gauss-Markov theorem prescribes to solve the weighted LS problem
\begin{equation}\label{eq: LSR}
\bar\bfx^\ast=\argmin_{\bfx\in\bbR^n}\nicefrac{1}{2}\|\bfA\bfx - \bfb\|_{\bfR^{-1}}^2,
\end{equation}
to find the best linear unbiased estimator (BLUE); see \cite{Zyskind1969GaussMarkov} and \cite[Chapter 10]{greene2003econometric}. Note that this reduces to \eqref{eq: LS} if the noise is Gaussian white since, in this case, $\bfR=\sigma^2\bfI$, where $\sigma$ is the standard deviation. 

In many situations, the covariance matrix $\bfR$ is not (fully) known but can be learned from the data. Indeed, often $\bfR$ is assumed to have a parametric form whose parameters can be estimated by solving an optimization problem. However, this usually requires computing  intermediate approximations of the solution $\bar{\bfx}^\ast$ for an intermediate set of parameters, so these solvers typically rely on nested cycles of iterations (inner for approximating $\bar\bfx^\ast$, outer for updating $\bfR$). 
Moreover, in some instances, $\bfR$ might depend explicitly on the noiseless measurements, as in methods that approximate Poisson noise and other heteroskedastic distributions;  see, e.g., \cite{Bardsley2026covariance} and \cite[Chapter 10]{greene2003econometric}. 
Similarly, the solution of \eqref{eq: linsys}, where the noise distribution has an underlying sparsity assumption, can be well recovered by considering a $p$-norm data fit, so that 
\begin{equation}\label{eq: LSp}
\bar{\bfx}^\ast=\argmin_{\bfx\in\bbR^n}\nicefrac{1}{p}\|\bfA\bfx-\bfb\|_p^p,
\end{equation} 
with $0<p\leq 1$. This is the case, for instance, of impulse or salt-and-pepper noise. More generally, the value of $p$ is linked to the assumption of the noise following a given Generalized
Error Distribution, so that, from a statistical viewpoint, solving \eqref{eq: LSp} corresponds to computing the maximum likelihood estimator; 
see \cite{benning2011error, Buccini2020GED}, and the references therein. 
Although sophisticated non-smooth and possibly non-convex optimization methods exist to handle these cases (see, e.g., \cite{dong2009, dong2013, dong2015}), we consider an Iteratively Reweighted Least Squares (IRLS) approach; see \cite[Section 4.5.2]{Bjo96a} and, specifically for regularized inverse problems, \cite{wohlberg2008lp, wohlberg2007tv}. 

IRLS methods can be regarded as specific instances of majorization-minimization (MM) methods, where a sequence of majorizing functions for \eqref{eq: LSp} is constructed at successive approximations of $\bar\bfx^\ast$, each obtained by minimizing the corresponding majorant. More specifically, given an approximation $\widetilde{\bfx}$ of a solution of \eqref{eq: linsys}, the objective function in \eqref{eq: LSp} can be approximated by considering the  objective function in \eqref{eq: LSR} with weights
\begin{equation}\label{eq:weights_lp}
\bfR^{-1} 
= (\bfW(\widetilde{\bfx}))^2= \diag{ 
\left((\bfA\widetilde{\bfx}-\bfb)^2+\tau^2\right)^{\frac{p-2}{2}}
}.
\end{equation}
Note that all the above operations on vectors are performed entry-wise and the smoothing parameter $\tau>0$ is considered to avoid possible divisions by zero caused by the lack of smoothness of the objective function in \eqref{eq: LSp} at zero. The choice of weights in \eqref{eq:weights_lp} allows to interpret \eqref{eq: LSR} as the tangent majorant of a smooth approximation of the original function at $\widetilde{\bfx}$ (ignoring additive and multiplicative constants). 
Classical IRLS solvers consider a sequence of LS problems of the form \eqref{eq: LSR}, and iteratively update the matrix $\bfR^{-1}$ by taking as $\widetilde{\bfx}$ the solution of the previous LS problem \eqref{eq: LSR} in the sequence, typically starting from $\bfR^{-1}= (\bfW(\bfzero))^2$. Since, in a general large-scale setting, each problem in the sequence is solved by CGLS or LSQR, IRLS solvers are actually inner-outer iterative schemes, whereby the weights are updated at each outer iteration. However, recently, more efficient IRLS formulations have been proposed, which enable to update the weights only after one iteration of the minimization algorithm has been performed (instead of a full solve). These solvers are typically based on `non-standard' (such as flexible \cite{chung2019, FKSIRW} or generalized \cite{huang2017majorization}) Krylov methods, and have been successfully considered to handle nonnegativity constraints in \eqref{eq: LS} \cite{buccini2020modulus} (possibly with the inclusion of covariance preconditioning, so to work with \eqref{eq: LSR} \cite{NNFCGLS}) 
and variational regularization methods with a $p$-norm fit-to-data term (like the one appearing in \eqref{eq: LSp}) and a $q$-norm regularization term; see also \cite{belgians, GS}. We emphasize that, among these `non-standard' Krylov methods, only the ones based on generalized Krylov subspaces can handle fit-to-data terms expressed in the $p$-norm.

This paper proposes new solvers for the weighted least squares problem \eqref{eq: LSR}, with $\bfR$ at least partially unknown and motivated by the statistics of the noise $\bfn$ in \eqref{eq: linsys}. 
The new solvers are inspired by the IRLS approach to \eqref{eq: LSR} and \eqref{eq: LSp} described above, but aim to avoid inner-outer iterations by incorporating iteration-dependent approximations of $\bfR$ on the fly. 
In order to achieve this, we introduce new instances of the flexible \cite{chung2019} and inexact \cite{Gazzola2021inexact} Golub-Kahan factorizations, linking them in specific cases and showing that, 
in the current setting, 
they are mathematically equivalent. 
Building on these new factorizations we introduce a range of flexible and inexact Krylov methods, based on the basic operations of approximation, projection and differentiation – whose details are explained later in this paper. We show that, in the current framework, it is possible to recover solvers that are mathematically equivalent to some solvers already available in the literature, as well as introduce some new ones. For most of the paper, the terms `flexible' and `inexact' are used interchangeably, and are mostly motivated by the solution of\eqref{eq: LSR} and \eqref{eq: LSp}. 

This paper is organized as follows. Section \ref{sect: factorizations} introduces the new flexible and inexact Golub-Kahan factorizations. Section \ref{sect:new_solvers} develops the theoretical foundations for, and presents, the proposed methods. In particular, in Section \ref{sect:PCGLS_PLSQR}, standard Krylov solvers for weighed LS problems are recalled, which serve as a baseline for the new solvers. Section \ref{sect:inexact} then presents and analyzes a new general framework for flexible and inexact solvers, providing a unified setting. Finally, Section \ref{sect: NumExp} reports numerical experiments on inverse problems in imaging, including test problems in image deblurring and computed 
tomography, showing that the proposed solvers are competitive with other IRLS-based alternatives.

\subsubsection*{Notations} In the following, $\bfI_d$ denotes the identity matrix of order $d$. Given a (bolded lower-case) vector $\bft\in\bbR^d$, $[\bft]_i$ denotes the $i$th entry of $\bft$. Given a (bolded upper-case) matrix $\bfC$, its entries are denoted by $[\bfC]_{i,j}$, its columns are denoted by $\bfc_j$, and its range is denoted by $\calR(\bfC)$. We denote functionals (i.e., functions from $\bbR^d$ to $\bbR$) with lower-case letters, e.g., $g(\bft)$. When such functionals are evaluated in a subspace of dimension $k$ (of $\bbR^d$), and typically at the $k$th iteration of an iterative solver, we use the notation $g_k(\bft)$, i.e. these are functions from $\bbR^k$ to $\bbR$ for a given parametrization. When such functional is iteration-dependent, at the $k$th iteration we use the notation $g^{(k)}(\bft)$. The above notations can be combined, e.g., $g_k^{(k)}(\bft)$. 

\section{Flexible and inexact Golub-Kahan factorizations}\label{sect: factorizations} Consider the matrix $\bfA \in \mathbb{R}^{m \times n}$ in \eqref{eq: linsys}, and an initial vector $\hat\bfu$, together with some SPD iteration-dependent preconditioning matrices $\bfR_1^{-1}, \dots, \bfR_{k+1}^{-1}$, with $k\ll\min\{m,n\}$ (ideally updated as the steps of the following factorizations proceed). 

First we consider the new Flexible Golub-Kahan (FGK) factorization:
\begin{equation}\label{eq:FGK_1}
\bfA\bfV_k=\bfU_{k+1}{\bfM}_k\,,\quad 
{\bfA}^{\top}\bfY_{k+1}=\bfV_{k+1}{\bfT_{k+1}}\,,
\end{equation}
where $\bfu_1 = \hat\bfu / \beta $ and $\beta = \|\hat\bfu\|_2
$, and 
\begin{itemize}
\item $\bfU_{k+1}\in\bbR^{m\times (k+1)}$ and $\bfV_{k+1}=[\bfV_{k},\bfv_{k+1}]\in\bbR^{n\times (k+1)}$ have orthonormal columns $\bfu_i$ and $\bfv_i$ ($i=1,\dots,k+1$), respectively; 
\item $\bfY_{k+1}=[\bfR_1^{-1}\bfu_1,\dots, \bfR_{k+1}^{-1}\bfu_{k+1}]\in\bbR^{m\times (k+1)}$;
\item ${\bfM}_k\in\bbR^{(k+1)\times k}$ is upper Hessenberg;
\item ${\bfT}_{k+1}\in\bbR^{(k+1)\times (k+1)}$ is upper triangular.
\end{itemize}
Note that, with respect to the flexible Golub-Kahan factorization introduced in \cite{chung2019}, we allow variable preconditioning on the right of $\bfA\t$, rather than on the right of $\bfA$. 

Second, we consider a new FGK version, mathematically equivalent to \eqref{eq:FGK_1}:
\begin{equation}\label{eq:FGK_2}
\bfA\bfV_k=\bfU_{k+1}{\bfM}_k\,,\quad 
{\bfA}^{\top}\bar\bfR_{k+1}\bfU_{k+1}=\bfV_{k+1}{\bfT_{k+1}}\,,
\end{equation}
where
\begin{equation}\label{defRkbar}
\bar \bfR_{k+1} = \sum_{i=1}^{k+1} \bfR^{-1}_i \bfu_i \bfu_i\t
\end{equation}
and where the matrices $\bfV_{k}$, $\bfV_{k+1}$, $\bfU_{k+1}$, ${\bfM}_k$ and $\bfT_{k+1}$ are those appearing in \eqref{eq:FGK_1}. A model implementation of the new flexible Golub-Kahan factorization in given in Algorithm \ref{alg: FGK}.

\begin{algorithm}
\caption{Flexible Golub-Kahan factorization} 
\label{alg: FGK}
\begin{algorithmic}[1]
	\REQUIRE{$\bfA\in\bbR^{m\times n}$, initial vector $\hat \bfu\in\bbR^m$}
\STATE $\bfu_1 = \hat\bfu /\beta$ with $\beta=\|\hat\bfu\|_2$
\STATE Initialize the precondtioner $\bfR_1^{-1}$
\STATE $\bfy_1 = \bfR_{1}^{-1}\bfu_1$
\STATE $\bfv_1=  \bfA\t \bfy_1 / [\bfT_{k+1}]_{1,1}$ with $[\bfT_{k+1}]_{1,1}=\| \bfA\t \bfy_1 \|_2$
\FOR {$i=1,\dots,k$}
\STATE $\bfu= \bfA\bfv_i$
\FOR{$j = 1, \ldots,i$}
  \STATE $\bfu=\bfu-[{\bfM}_k]_{j,i}\bfu_j$, with $[{\bfM}_k]_{j,i}=\bfu\t \bfu_j$
\ENDFOR
\STATE $\bfu_{i+1}=\bfu/[{\bfM}_k]_{i+1,i}$, with $[{\bfM}_k]_{i+1,i}=\|\bfu\|_2$
\STATE Update the preconditioner $\bfR_{i+1}^{-1}$
\STATE $\bfy_{i+1} = \bfR_{i+1}^{-1}\bfu_{i+1}$
\STATE $\bfv = \bfA\t \bfy_{i+1}$
\FOR{$j = 1, \ldots,i$}
  \STATE $\bfv=\bfv-[\bfT_{k+1}]_{j,i+1}\bfv_j$, with $[\bfT_{k+1}]_{j,i+1}=\bfv\t \bfv_j$
\ENDFOR
\STATE\label{line:vupdate} $\bfv_{i+1}=\bfv / [\bfT_{k+1}]_{i+1,i+1}$ with $[\bfT_{k+1}]_{i+1,i+1}=\|\bfv\|_2$ 
\ENDFOR\\
\ENSURE $\bfV_{k+1},\,\bfU_{k+1},\, \bfY_{k+1}, \,\bfM_k,\,\bfT_{k+1}$
\end{algorithmic}
\end{algorithm}

Finally, we consider the following inexact (iGK) version of \eqref{eq:FGK_1}, \eqref{eq:FGK_2}: 
\begin{equation}\label{eq:iGK}
\bfA\bfV_k = \bfU_{k+1}\bfM_{k},\quad
(\bfA+{\calE}_{k+1}^{(k)})\t\bfR_{k+1}^{-1}\bfU_{k+1}=\bfV_{k+1}\bfT_{k+1}\,,
\end{equation}
where ${\calE}_{k+1}^{(k)}$ is a matrix incorporating some iteration-dependent inexactness in $\bfA\t\bfR_{k+1}^{-1}$, defined as:
\begin{equation} \label{eq:errors1}
{\calE}_{k+1}^{(k)}=\bfR_{k+1}\left(\sum_{i=1}^{k+1}\bfu_i\bfu_i\t(\bfE_i^{(k)})\right)\bfA\,, 
\quad \text{with} \quad 
\bfE_i^{(k)}:=\bfR_i^{-1}-\bfR_{k+1}^{-1}\,.    
\end{equation}
Note that, in the above definition, we have used the facts that both $\bfE_i^{(k)}$ and $\bfR_{k+1}$ are symmetric, and that the latter is positive definite. 
All the matrices $\bfE_i^{(k)}$, 
$i = 1,\dots,k+1$, change at each iteration $k$, leading to an iteration-dependent ${\calE}_{k+1}^{(k)}$. The matrices $\bfV_{k}$, $\bfV_{k+1}$, $\bfU_{k+1}$, ${\bfM}_k$ and $\bfT_{k+1}$ are again those appearing in \eqref{eq:FGK_1}.

A fundamental conceptual difference between flexible factorizations of the kind \eqref{eq:FGK_1}, \eqref{eq:FGK_2} and the inexact factorization \eqref{eq:iGK} is that, in the former, the matrices $\bfR_i^{-1}$ are just, formally, iteration-dependent preconditioners while, in the latter, one needs to define inexactness and errors with respect to an ideally exact matrix. 
Fixing some notion of inexactness is useful also to recast this problem in the framework of the inexact Krylov solvers presented in \cite{Gazzola2021inexact}. Coherently with \cite{Gazzola2021inexact}, for reasons that will be clear later (and related to the IRLS framework introduced in \Cref{sec: Intro}), we regard the SPD matrices $\bfR_1^{-1},\dots,\bfR_{k}^{-1}$ as inexact versions of $\bfR_{k+1}^{-1}$; this is also reflected in the fact that, in \eqref{eq:iGK}, $\bfR_{k+1}^{-1}$ appears as an exact right preconditioner for $\bfA\t$. With respect to the inexact factorizations introduced in \cite{Gazzola2021inexact}, \eqref{eq:iGK} does not allow inexactness in the matrix-vector products with $\bfA$ and includes the preconditioner $\bfR_{k+1}^{-1}$. However, factorizations \eqref{eq:FGK_1}, \eqref{eq:FGK_2} and \eqref{eq:iGK} (given the errors defined in \eqref{eq:errors1}) are all mathematically equivalent and, in the following, we most often refer to the approximation subspace
\begin{equation}\label{def: RV_k}
\begin{split}
\calR(\bfV_k)&= 
\calK_k(\bfA\t\bar\bfR_{k+1}\bfA,\bfA\t\bar\bfR_{k+1}\hat{\bfu})\\
&= 
\calK_k((\bfA+{\calE}_{k+1}^{(k)})\t\bfR_{k+1}^{-1}\bfA,(\bfA+{\calE}_{k+1}^{(k)})\t\bfR_{k+1}^{-1}\hat{\bfu})
\end{split}
\end{equation}
associated to any of them as flexible/inexact Krylov subspace, and the solvers defined by imposing that the $k$th approximate solution $\bfx_k$ belongs to $\calR(\bfV_k)$ and some constraints on the residual as flexible/inexact Krylov methods. 

Finally, for both the factorizations \eqref{eq:FGK_2} and \eqref{eq:iGK} (but not \eqref{eq:FGK_1}), we can derive the following flexible/inexact Lanczos-like relationships, respectively:
\begin{equation}\label{eq:iLanczos}
\begin{array}{rcl}
\bfA\t\bar{\bfR}_{k+1}\bfA\bfV_k & = & \bfV_{k+1}\widehat{\bfH}_k\\
    (\bfA + {\calE}_{k+1}^{(k)})\t\bfR_{k+1}^{-1}\bfA\bfV_k & =& \bfV_{k+1}\widehat{\bfH}_k
    \end{array},
\quad\mbox{where}\quad\widehat{\bfH}_k:=\bfT_{k+1}\bfM_k\,.
\end{equation}

Recall that $\bar{\bfR}_{k+1}$ and ${\calE}_{k+1}^{(k)}$ are two alternative ways of expressing the same error or inexactness, and are defined in \eqref{defRkbar} and \eqref{eq:errors1}, respectively. 

\begin{remark}\label{rem: iArnoldi}
Analogously to what already observed in \cite{Gazzola2021inexact}, \eqref{eq:iLanczos} is not equivalent to applying inexact Arnoldi to the  iteration-dependent system matrix $\bfA\t\bfR_{i}^{-1}\bfA$ approximating the exact system matrix $\bfA\t\bfR_{k+1}^{-1}\bfA$. Indeed, the latter compactly reads:
\[
\left(\bfA\t \bfR_{k+1}^{-1} \bfA + \widetilde{\calE}^{(k)}_{k}  \right) \bfV_k =
\bfV_{k+1} \bfH_{k}, \quad \text{where}
\quad 
\widetilde{\calE}^{(k)}_{k} = \bfA\t \sum_{i=1}^{k} \bfE^{(k)}_{i} \bfA \bfv_{i} \bfv_{i}\t\,, \]
here the matrix $\bfH_k\in\bbR^{(k+1)\times k}$ is upper Hessenberg and the matrices $\bfE_i^{(k)}$ are defined in \eqref{eq:errors1}. The above factorization can be easily derived starting from an inexact Arnoldi relationship of the kind:
\[
    [\bfA\t \bfR_{1}^{-1} \bfA \bfv_{1},\dots,\bfA\t \bfR_{k}^{-1} \bfA \bfv_{k}] = \bfV_{k+1}\bfH_k\,,
\]
where the matrix on the left can be spelled out as:
\begin{eqnarray*}
    \bfA\t \bfR_{k+1}^{-1}\bfA \bfV_{k} + \bfA\t[\bfE_1^{(k)} \bfA \bfv_1,\dots,\bfE_{k}^{(k)} \bfA \bfv_{k}]
    =  \bfA\t \bfR_{k+1}^{-1}\bfA \bfV_{k} + \widetilde{\calE}^{(k)}_{k} \bfV_{k}.
\end{eqnarray*}   
\end{remark}

We conclude this section by emphasizing that, even if the factorizations hereby introduced allow any initial vector $\hat\bfu\in\bbR^m$ to be used, in the following we always assume that $\hat\bfu$ is related to $\bfr_0=\bfb - \bfA \bfx_0$, where $\bfx_0$ is an initial guess for the solution of \eqref{eq: linsys}. Moreover, although we refer to the matrices $\bfR_i^{-1}$, $i=1,\dots,k+1$, as preconditioners, in the following they are not intended with the classical purpose of accelerating the convergence of iterative solvers; they rather stem from the IRLS method applied to \eqref{eq: LSR} or \eqref{eq: LSp} and, as highlighted in \Cref{sec: Intro}, their purpose is to improve the quality of the approximations to the solution of \eqref{eq: linsys}. Still in this setting, the matrices $\bfR_i^{-1}$, $i=1,\dots,k+1$, are updated using residual vectors, noting that, for the considered solvers, the solutions and residuals at the $i$th iteration, $i\leq k$, can be computed only after $\bfv_{i+1}$ in line \ref{line:vupdate} of Algorithm \ref{alg: FGK} has been computed.

\section{Standard, flexible and inexact solvers for general data-fitting problems}\label{sect:new_solvers}
We slightly reformulate \eqref{eq: LSR} as the problem of computing $\bar{\bfx}^\ast=\bfx_0+\bfx^\ast$, where a correction 
\begin{equation}\label{eq: FitPb}
\bfx^\ast=\argmin_{\bfx\in\bbR^n}g(\bfx)
\end{equation}
is computed to form an improved approximation $\bar\bfx^\ast$ to $\bfx\true$ in \eqref{eq: linsys}, and where
\begin{equation}\label{eq: g}
\begin{array}{rl}
g(\bfx) &=\nicefrac{1}{2}\left(\bfx\t\bfA\t\bfR^{-1}\bfA\bfx - 
2\bfx\t\bfA\t\bfR^{-1}\bfr_0 
+\bfr_0\t\bfR ^{-1}\bfr_0\right)\vspace{0.1cm}\\ 
&= \nicefrac{1}{2}\|\bfA\bfx - \bfr_0\|_{\bfR^{-1}}^2\,.
\end{array}
\end{equation}
We will refer to $g(\bfx)$ as \textit{exact functional}. Problem \eqref{eq: LSR} is obtained by setting $\bfx_0=\bfzero$ in \eqref{eq: FitPb}. 

If the matrix $\bfR$ is given, then one can apply the mathematically equivalent preconditioned CGLS or LSQR to efficiently solve this problem. We will dwell on this case in Section \ref{sect:PCGLS_PLSQR}. However, as explained in Section \ref{sec: Intro}, there are relevant situations where $\bfR$ is (at least partially) unknown, and we have to update $\bfR$ based on current approximations to $\bfx\true$. We will dwell on this case in Section~\ref{sect:inexact}.

In describing all the solvers we rely on the overarching framework for Krylov solvers introduced in \cite{Eiermann_Ernst_2001}. That is, we classify them as either orthogonal `residual' (O-`R') or minimal `residual' (M-`R') methods, with the term `residual' to be broadly intended. More precisely, at the $k$th iteration of each solver, given an approximation subspace $\check{\calC}_k$ for the correction $\bfx^\ast$, an approximation subspace $\check{\calW}_k$ for the residual correction, and a possible subspace of constraints $\check{\calV}_k$ (both dependent on $\check{\calC}_k$), we determine a residual update $\check{\bfw}_k^{\text{\scriptsize{O}}}$ or $\check{\bfw}_k^{\text{\scriptsize{M}}}$ belonging to $\check{\calW}_k$ such that
\[
\begin{array}{ccll}
\check{\bfr}_0-\check{\bfw}_k^{\text{\scriptsize{O}}}&\perp& \check{\calV_k}, &\quad\mbox{for O-`R' methods,}\vspace{0.1cm}\\
\|\check{\bfr}_0-\check{\bfw}_k^{\text{\scriptsize{M}}}\|&=&\min_{\check{\bfw}\in\check{\calW_k}}\|\check{\bfr}_0-\check{\bfw}\|_2, &\quad\mbox{for M-`R' methods.}\\
\end{array}
\]
Equivalently, $\check{\bfw}_k^{\text{\scriptsize{O}}}$ is the oblique projection of $\check\bfr_0$ onto $\check{\calW}_k$, orthogonal to $\check{\calV}_k$ and $\check{\bfw}_k^{\text{\scriptsize{M}}}$ is the orthogonal projection of $\check\bfr_0$ onto $\check{\calW}_k$ (so that $\check{\calV}_k=\check{\calW}_k$). 
In Table \ref{tab:meth_summary} we provide an upfront summary of all the methods considered in this section, specifying the spaces $\check{\calC}_k$, $\check{\calW}_k$, and $\check{\calV}_k$, and the definition of the `residual' $\check{\bfr}_0$. Note that the naming conventions used in the first column are motivated and detailed in the next subsections.
\begingroup
\renewcommand{\arraystretch}{1.5} 
\setlength{\tabcolsep}{2pt} 
\begin{table}
\caption{Summary of all the methods considered in \Cref{sect:new_solvers}, using the compact notation introduced in the first row, with $\simeq$ indicating that a stated property only approximately holds. Among the solvers, the following are mathematically equivalent: LSQR and CGLS; DAP, DPA and PDA; APD, PAD and ADP.}
    \label{tab:meth_summary}
    \footnotesize
    \centering
    \begin{tabular}{|c|c|c|c|c|c|}\hline
    \textsc{Notations}:&\multicolumn{5}{l|}{ $\whAEk=(\bfA+\calEk)\t\Rki\bfA,\:$ $\AEkT=(\bfA+\calEk)\t\Rki,\:$ $\AcalE=\bfA+\calEk$}\\\hline
    \textbf{solver} & \textbf{category} & $\check{\calC}_k$ & $\check{\calW}_k$ & $\check{\calV}_k$ & $\check{\bfr}_0$\\\hline
    LSQR     & M-$g$ & $\calK_k(\bfA\t\bfR^{-1}\bfA,\bfA\t\bfR^{-1}\bfr_0)$ & $\bfR^{-\nicefrac{1}{2}}\bfA\check{\calC}_k$ & $\check{\calW}_k$ & $\bfR^{-\nicefrac{1}{2}}\bfr_0$\\\hline
    CGSL & O-$\nabla g$ &$\calK_k(\bfA\t\bfR^{-1}\bfA,\bfA\t\bfR^{-1}\bfr_0)$ & $\bfA\t\bfR^{-1}\bfA\check{\calC}_k$ &
    
$\check{\calC}_k$ & $\bfA\t\bfR^{-1}\bfr_0$\\\hline
DAP & O-$\widehat{\nabla g^{(k)}}$ & $\calK_k(\whAEk,\AEkT\bfr_0)$ & 
$\whAEk\check{\calC}_k$ & $\check{\calC}_k$ & $\AEkT\bfr_0$\\\hline
DPA &$\simeq$ O-${\nabla g^{(k)}}$ & $\calK_k(\whAEk,\AEkT\bfr_0)$ &
$\bfA\t\Rki\bfA\check{\calC}_k$ & $\check{\calC}_k$ & $\bfA\t\Rki\bfr_0$\\\hline
PDA & $\simeq$ M-$g^{(k)}$ & 
$\calK_k(\whAEk,\AEkT\bfr_0)$ & $\bfR_{k+1}^{-\nicefrac{1}{2}}\bfA\check{\calC}_k$ & 
$\bfR_{k+1}^{-\nicefrac{1}{2}}\AcalE\check{\calC}_k$ & $\bfR_{k+1}^{-\nicefrac{1}{2}}\bfr_0$ \\\hline
DAP-LSMR & M-$\widehat{\nabla g^{(k)}}$ & $\calK_k(\whAEk,\AEkT\bfr_0)$ & 
$\whAEk\check{\calC}_k$ & $\whAEk\check{\calC}_k$ & $\AEkT\bfr_0$\\\hline
APD, PAD & M-$\bar{g}^{(k)}$ & $\calK_k(\bfA\t\bRk\bfA,\bfA\t\bRk\bfr_0)$ & $\bfA\t\bRSk\bfA\check{\calC}_k$ & $\check{\calC}_k$ & $\bfA\t\bRk^{(S)}\bfr_0$\\\hline
ADP & M-$\nabla\bar{g}^{(k)}$ & 
$\calK_k(\bfA\t\bRk\bfA,\bfA\t\bRk\bfr_0)$ & $\bfA\t\bRSk\bfA\check{\calC}_k$ & $\check{\calC}_k$ & $\bfA\t\bRk^{(S)}\bfr_0$\\\hline
    \end{tabular}
\end{table}
\endgroup 

\subsection{Exact Krylov solvers for the exact functional}\label{sect:PCGLS_PLSQR}

Solving problem~\eqref{eq: FitPb} for a fixed known $\bfR$ can be done either using preconditioned CGLS, or preconditioned LSQR, which are mathematically equivalent; see \cite[Chapter 4]{Bjo15}. In this setting, the `preconditioner' $\bfR^{-1}$ is equivalently applied to the right of $\bfA\t$ or to the left of $\bfA$.

In particular, preconditioned CGLS solves the normal equations associated to \eqref{eq: LSR} using conjugate gradient. We focus on the interpretation of CGLS as a `first differentiate, then project' approach, i.e., the gradient of $g(\bfx)$ in \eqref{eq: g} or, equivalently, the normal equation residual, is projected onto a Krylov subspace. Specifically, at the $k$th iteration, CGLS computes an update for the solution $\bar{\bfx}_k = \bfx_0+\bfx_k$, where $\bfx_k \in \calR(\bfV_k)=\calK_k(\bfA\t\bfR^{-1}\bfA,\bfA\t\bfR^{-1}\bfr_0)$, i.e., $\bfx_k =\bfV_k\bfs_k$ and $\bfs_k\in\bbR^k$ satisfies
\begin{equation}\label{eq:NE1_Galerkin}
\nabla g(\bfx_k)=\bfA\t \bfR^{-1} (\bfA\bfx_k-\bfr_0) \perp \mathcal{R}(\bfV_{k}) 
\,\Longleftrightarrow\,\bfV_{k}\t\bfA\t 
\bfR^{-1}
(\bfA\bfV_k\bfs_k-\bfr_0)=\bfzero\,.
\end{equation}

According to the framework in \cite{Eiermann_Ernst_2001}, preconditioned CGLS is an `orthogonal residual' (OR) method applied to the normal equations of \eqref{eq: LSR}, so we more properly refer to it as an `orthogonal normal equation residual' (O-$\nabla g$) method. 

As such, preconditioned CGLS performs a projection procedure, which determines $\bfA\t\bfR^{-1}\bfA\bfx_k$ (i.e., the image of the correction $\bfx_k$ under $\bfA\t\bfR^{-1}\!\bfA$, belonging to $\bfA\t\bfR^{-1}\!\bfA\calK_k(\bfA\t\bfR^{-1}\!\bfA,\bfA\t\bfR^{-1}\!\bfr_0)$) by an oblique projection of $\bfA\t\bfR^{-1}\bfr_0$ (orthogonal to $\calK_k(\bfA\t\bfR^{-1}\bfA,\bfA\t\bfR^{-1}\bfr_0)$). 
In the framework presented in \cite{Bjo15, saad2003iterative}, preconditioned CGLS is regarded as an orthogonal projection method for the preconditoned normal equations, since it uses $\calK_k(\bfA\t\bfR^{-1}\bfA,\bfA\t\bfR^{-1}\bfr_0)$ as both approximation subspace for the solution $\bfx_k$ and constraint subspace for the normal equation residual.  

The mathematically equivalent preconditioned LSQR method can be interpreted as a `first project, then differentiate' approach. In this case, we directly minimize $g(\bfx)$ in \eqref{eq: g} over $\calR(\bfV_k)=\calK_k(\bfA\t\bfR^{-1}\bfA,\bfA\t\bfR^{-1}\msl{\bfr_0})$, i.e., we consider the restriction $\bfx=\bfV_k\bfs$ and optimize over the coefficients $\bfs\in\bbR^k$. In formulas, we compute
\begin{equation}\label{eq:NE1_PetrovGalerkin}
\bfs_k=\argmin_{\bfs\in\bbR^k}\underbrace{\nicefrac{1}{2}\|\bfA\bfV_k\bfs-\bfr_0\|_{\bfR^{-1}}^2}_{=:g_k(\bfs)}\quad\Longleftrightarrow\quad\underbrace{\bfV_k\t\bfA\t\bfR^{-1}(\bfA\bfV_k\bfs_k-\bfr_0)}_{=\,\nabla g_k(\bfs_k)}=\bfzero\,,
\end{equation}
and take $\bfx_k=\bfV_k\bfs_k$. 
According to the framework in \cite{Eiermann_Ernst_2001}, preconditioned LSQR is a `minimal residual' (MR) method or, equivalently, it is based on a projection procedure that determines $\bfR^{-\nicefrac{1}{2}}\bfA\bfx_k$ by performing an orthogonal projection of $\bfR^{-\nicefrac{1}{2}}\bfr_0$ onto 
$\bfR^{-\nicefrac{1}{2}}\bfA\calK_k(\bfA\t\bfR^{-1}\bfA,\bfA\t\bfR^{-1}\bfr_0)$. \msl{Equivalently, $\bfA\bfx_k$ is the  orthogonal projection of $\bfr_0$ onto $\bfA\calK_k(\bfA\t\bfR^{-1}\bfA,\bfA\t\bfR^{-1}\bfr_0)$ in the $\bfR^{-1}$-weighted inner product. } 
However, since preconditioned LSQR operates on the objective function $g(\bfx)$ in \eqref{eq: g} (i.e., a weighted residual norm) rather than on the residual norm, we more properly refer to it as a `minimal $g$' (M-$g$) method. 
In the framework presented in \cite{Bjo15, saad2003iterative}, whereby preconditioned LSQR is regarded as an oblique projection method \msl{for the preconditioned system} having  $\calK_k(\bfA\t\bfR^{-1}\bfA,\bfA\t\bfR^{-1}\bfr_0)$ as approximation subspace for the solution $\bfx_k$ and $\bfA\calK_k(\bfA\t\bfR^{-1}\bfA,\bfA\t\bfR^{-1}\bfr_0)$ as constraint subspace for the preconditioned residual $\bfR^{-1}(\bfA\bfx_k-\bfr_0)$. 
Following projection, differentiation is performed in \eqref{eq:NE1_PetrovGalerkin} to impose the optimality conditions, leading to the solution of a linear system with SPD coefficient matrix $\bfV_k\t\bfA\t\bfR^{-1}\bfA\bfV_k$. 
We refer to $g_k(\bfs)$ in \eqref{eq:NE1_PetrovGalerkin} as the \textit{$\calR(\bfV_k)$-restricted exact functional}. 

In summary, for both preconditioned CGLS and LSQR, at iteration $k$ we obtain an approximate solution $\bfx_k=\bfV_k \bfs_k$ of \eqref{eq: FitPb} by computing the coefficients $\bfs_k$ that equivalently solve \eqref{eq:NE1_Galerkin} and \eqref{eq:NE1_PetrovGalerkin}, respectively. Let us consider the relations in \eqref{eq:FGK_1}, \eqref{eq:FGK_2} for fixed $\bfR_i^{-1}=\bfR^{-1}$, which simplify to 
\begin{equation}\label{eq: pGKB}
\bfA\bfV_k=\bfU_{k+1}\bfB_{k+1,k},\quad \bfA\t\bfR^{-1}\bfU_{k+1}=\bfV_{k+1}\bfB_{k+1}\t\,,
\end{equation}
where $\bfV_{k}$, $\bfV_{k+1}$, $\bfU_{k+1}$ are as in \eqref{eq:FGK_1}, \eqref{eq:FGK_2}, $\bfB_{k+1}\in\bbR^{(k+1)\times (k+1)}$ is lower bidiagonal and $\bfB_{k+1,k}$ is obtained by removing the last column of $\bfB_{k+1}$. 
Exploiting \eqref{eq: pGKB} and taking $\alpha_1:=[\bfB_{k+1}]_{1,1}$ and $\beta:=\|\bfr_0\|_{2}$, 
both the linear systems on the left of \eqref{eq:NE1_Galerkin} and \eqref{eq:NE1_PetrovGalerkin} simplify to
\[
{\bfB}_{k+1,k}\t{\bfB}_{k+1,k} \bfs_k = \alpha_1\beta\bfe_1\,.
\]

Note, however, that common implementations of preconditioned CGLS and LSQR do not strictly follow the framework outlined above, with the former relying on three-term recurrences and the latter relying of smart updates of the QR factorization of ${\bfB}_{k+1,k}$; see \cite[Chapter 4]{Bjo15} for more details. In the following, we will refer to solvers that first differentiate  and then project as `CGLS-like'; we will refer to solvers that first project and then differentiate as `LSQR-like'. Coherently with the terminology already adopted in this subsection, whenever we are concerned with the minimization of a functional $g(\bfx)$, in the following we will often refer to $\nabla g(\bfx)$ as `normal equations residual'.

\subsection{Mixing and matching differentiation, projection, approximation}\label{sect:inexact} This section concerns the approximation of the solution of problem \eqref{eq: FitPb}, with an (at least partially) unknown $\bfR$ in \eqref{eq: g} that is updated at each iteration of a solver for \eqref{eq: FitPb}. To handle this situation, we present a new framework whereby flexible/inexact projection methods can be formulated: this encompasses some known solvers, and sets the stage for the introduction of new ones. Building on the basic operations of `projection' (P) and `differentiation' (D) already introduced in Section~\ref{sect:PCGLS_PLSQR}, we also consider the effect of approximation (A): flexible/inexact Krylov methods differ in the order in which such operations are performed, so that we name each method after permutations of the letters D, A, P. 
In particular, the relative order of A and D will give rise to different methods, while the relative order between D and P will characterize mathematically equivalent CGLS-like and LSQR-like methods.  

\subsubsection{DAP, DPA and PDA methods, i.e., `classical' inexact Krylov solvers}\label{sect: inexact1}
As we consider $\bfR_k^{-1}$ being iteration-dependent, at the $k$th iteration of a solver for \eqref{eq: FitPb} we may consider finding an approximate minimizer of the iteration-dependent objective function
\begin{equation}\label{eq: gk_exact}
\begin{array}{rl}
g^{(k)}(\bfx) &= \nicefrac{1}{2}\left(\bfx\t\bfA\t\bfR_{k+1}^{-1}\bfA\bfx - 
2\bfx\t\bfA\t\bfR_{k+1}^{-1}\bfr_0 
+\bfr_0\t\bfR_{k+1} ^{-1}\bfr_0\right)\vspace{0.1cm}\\ 
&= \nicefrac{1}{2}\|\bfA\bfx - \bfr_0\|_{\bfR_{k+1}^{-1}}^2,
\end{array}
\end{equation}
which we call the \textit{exact functional at the $k$th iteration}. This choice also matches the IRLS setting mentioned in \Cref{sec: Intro}, as we would define as current reweighted least squares problem the one with weights computed with respect to the most recent approximation to $\bfx^\ast$ in \eqref{eq: FitPb}. 

We start by considering a method that first differentiates (D) the objective function \eqref{eq: gk_exact}, then approximates (A) its gradient and eventually projects (P) it onto an approximation subspace, i.e., a DAP method. That is, we compute
\begin{equation}\label{eq:inexact_gradient}
\nabla g^{(k)}(\bfx) = \bfA\t \bfR_{k+1}^{-1} (\bfA\bfx-\bfr_0)\; \approx\; \widehat{\nabla g^{(k)}}(\bfx):=(\bfA + {\calE}^{(k)}_{k+1})\t \bfR_{k+1}^{-1} (\bfA\bfx-\bfr_0)
\end{equation}
and then apply a projection method involving the approximate gradient (or inexact normal equations residual) $\widehat{\nabla g^{(k)}}(\bfx)$. 
Specifically, at the $k$th iteration of such solver, we compute a solution \begin{equation}\label{def:xsubspinex}
\bfx_k \in \calR(\bfV_k)=\calK_k((\bfA+{\calE}_{k+1}^{(k)})\t\bfR_{k+1}^{-1}\bfA,(\bfA+{\calE}_{k+1}^{(k)})\t\bfR_{k+1}^{-1}\bfr_0),
\end{equation}
such that
\begin{equation}\label{eq:NE1_Galerkin_inexact}
\widehat{\nabla g^{(k)}}(\bfx_k)=(\bfA + {\calE}^{(k)}_{k+1})\t \bfR_{k+1}^{-1} (\bfA\bfx_k-\bfr_0) \perp \mathcal{R}(\bfV_{k}),
\end{equation}
or, equivalently, taking $\bfx_k=\bfV_k\bfs_k$,
\begin{equation}\label{iCGLS}
\bfV_k\t(\bfA+{\calE}^{(k)}_{k+1})\t\bfR_{k+1}^{-1}\bfA\bfV_k\bfs_k=\bfV_k\t(\bfA+{\calE}^{(k)}_{k+1})\t\bfR_{k+1}^{-1}\bfr_0.
\end{equation}
In practice, using the factorizations \eqref{eq:iGK}, the left-hand-side of the above equation reads
\begin{eqnarray*}
\bfV_k\t(\bfA+\calEk)\t\bfR_{k+1}^{-1}\bfU_{k+1}{\bfM}_k\bfs_k = \bfV_k\t\bfV_{k+1}\bfT_{k+1}{\bfM}_k\bfs_k
=[\bfI_k, \bfzero]\bfT_{k+1}{\bfM}_k\bfs_k,
\end{eqnarray*}
while the right-hand-side reads
\begin{eqnarray}\label{eq:projrhs}
\bfV_{k}\t(\bfA+\calEk)\t\bfR_{k+1}^{-1}\beta\bfu_1 
=\beta\bfV_k\t\bfV_{k+1}\bfT_{k+1}\bfe_1=\beta t_{1,1}\bfe_1,
\end{eqnarray}
leading to the following projected linear system
\begin{equation}\label{eq:projected_classic}
{\bfT}_{k,k+1}\bfM_k\bfs_k=\beta t_{1,1}\bfe_1,
\end{equation}
where ${\bfT}_{k,k+1}$ is $\bfT_{k+1}$ without its last row, $\beta=\|\bfr_0\|_{2}$ and $t_{1,1}=[\bfT_{k+1}]_{1,1}$. The DAP method is a CGLS-like method and, indeed, the above formulation coincides with the inexact CGLS (iCGLS) method proposed in \cite{Gazzola2021inexact}. In the framework of \cite{Eiermann_Ernst_2001}, such DAP method may be regarded as an `orthogonal inexact normal equation residual' (O-$\widehat{\nabla g^{(k)}}$) method that determines $(\bfA + {\calE}^{(k)}_{k+1})\t \bfR_{k+1}^{-1}\bfA\bfx_k$ by performing an oblique projection of $(\bfA + {\calE}^{(k)}_{k+1})\t \bfR_{k+1}^{-1}\bfr_0$ onto $(\bfA+\calE_{k+1}^{(k)})\t\bfR_{k+1}^{-1} \bfA\calR(\bfV_k)$, orthogonal to  $\calR(\bfV_k)$. 
In the framework of \cite{Bjo15, saad2003iterative} such DAP method is an orthogonal projection method having $\calR(\bfV_k)=\calK_k((\bfA + {\calE}^{(k)}_{k+1})\t \bfR_{k+1}^{-1}\bfA,(\bfA + {\calE}^{(k)}_{k+1})\t \bfR_{k+1}^{-1}\bfr_0)$ as both approximation subspace for the solution $\bfx_k$ and constraint subspace for the inexact normal equation residual $\widehat{\nabla g^{(k)}}(\bfx_k)$. 

Equivalently, after differentiating the functional \eqref{eq: gk_exact}, we can compute an approximate projection of $\nabla g^{(k)}(\bfx)$ onto the inexact Krylov subspace $\calR(\bfV_k)$, i.e., we differentiate (D), project (P) and then approximate (A), to obtain a DPA method. In formulas, we still take $\bfx_k=\bfV_k\bfs_k\in\calR(\bfV_k)$ as in \eqref{def:xsubspinex} but, differently from \eqref{eq:NE1_Galerkin_inexact}, determine $\bfs_k\in\bbR^k$ by imposing the following orthogonality condition on the normal equation residual $\nabla g^{(k)}(\bfx)$ (associated to the minimization of the iteration-dependent functional $g^{(k)}(\bfx)$ in  \eqref{eq: gk_exact})
\begin{equation}\label{eq:NE2_Galerkin_inexact}
\underbrace{\bfA\t \bfR_{k+1}^{-1} (\bfA\bfx_k-\bfr_0)}_{=\nabla g^{(k)}(\bfx_k)}=\bfA\t\Rki(\bfA\bfV_k\bfs_k-\bfr_0) \perp \mathcal{R}(\bfV_{k})\,.
\end{equation}
Exploiting again the factorizations \eqref{eq:iGK}, with algebraic manipulations similar to the ones performed for the DAP method, we obtain
\[
({\bfT}_{k,k+1}\bfM_k - \underbrace{\bfV_k\t({\calE}^{(k)}_{k+1})\t \bfR_{k+1}^{-1}\bfA\bfV_k}_{=:\widehat{\bfE}_{k+1}})\bfs_k=\beta t_{1,1}\bfe_1\,.
\]
Getting rid of the term $\widehat{\bfE}_{k+1}$ in the above parenthesis (i.e., considering an inexact projection), we recover problem \eqref{eq:projected_classic}. Such DPA method is still a CGLS-like method that, in the framework of \cite{Eiermann_Ernst_2001}, determines $\bfA\t\Rki\bfA\bfx_k$ by performing an inexact oblique projection of $\bfA\t\Rki\bfr_0$ onto $\bfA\t\Rki\bfA\calR(\bfV_k)$, orthogonal to $\calR(\bfV_k)$, and therefore can be regarded as an `approximate orthogonal normal equation residual' ($\simeq$ O-$\nabla g^{(k)}$) method. 

Finally, we may again take $\bfx=\bfV_k\bfs$, with $\bfs\in\bbR^{k}$ and $\bfV_k$ as in \eqref{def:xsubspinex}, in the iteration-dependent exact functional $g^{(k)}(\bfx)$ in \eqref{eq: gk_exact}, 
and consider a projection (P) analogous to the one underlying the M-$g$ method in Section \ref{sect:PCGLS_PLSQR}. That is, we take
\[
\bfs_k=\argmin_{\bfs\in\bbR^{k}}g_k^{(k)}(\bfs),\quad\mbox{where}\quad
g_k^{(k)}(\bfs)=\nicefrac{1}{2}\|\bfA\bfV_k\bfs-\bfr_0\|_{\bfR_{k+1}^{-1}}^2\,,
\]
which we call \emph{restricted exact functional at the $k$th iteration}. We then differentiate (D) to impose the optimality condition, to compute $\bfs_k\in\bbR^k$ by solving
\begin{equation}\label{eq: projNEgk}
\bfV_k\t\bfA\t\bfR_{k+1}^{-1}\bfA\bfV_k\bfs_k-\bfV_k\t\bfA\t\bfR_{k+1}^{-1}\bfr_0=\bfzero\,.
\end{equation}
Similarly to the DPA method, we rewrite
\[
(\underbrace{\bfV_k\t(\bfA+{\calE}^{(k)}_{k+1})\t\bfR_{k+1}^{-1}\bfA\bfV_k}_{{\bfT}_{k,k+1}\bfM_k} -
\underbrace{\bfV_k\t({\calE}^{(k)}_{k+1})\t \bfR_{k+1}^{-1}\bfA\bfV_k}_{=:\widehat{\bfE}_{k+1}})\bfs_k=\beta t_{1,1}\bfe_1
\]
and approximate (A) the above equation by dropping the term $\widehat{\bfE}_{k+1}$, 
eventually recovering again the projected problem \eqref{eq:projected_classic}. 
Therefore the PDA method so defined is mathematically equivalent to the previous DAP and DPA methods, and can be regarded as the LSQR-like version of them, or as an `approximate minimal $g^{(k)}$' ($\simeq$ M-$g^{(k)}$) method. 
In the framework of \cite{Eiermann_Ernst_2001}, the quantity $\bfR_{k+1}^{-\nicefrac{1}{2}} \bfA\bfx_k$ is now obtained as an oblique (rather than orthogonal) projection of $\bfR_{k+1}^{-\nicefrac{1}{2}} \bfr_0$ onto $\bfR_{k+1}^{-\nicefrac{1}{2}} \bfA \mathcal{R}(\bfV_k)$, orthogonal to $\bfR_{k+1}^{-\nicefrac{1}{2}}(\bfA +{\calE}^{(k)}_{k+1})\mathcal{R}(\bfV_k)$. Equivalently, $\bfA\bfx_k$ is now obtained as an  oblique (rather than orthogonal) projection of $\bfr_0$ onto $\bfA \mathcal{R}(\bfV_k)$, orthogonal to $(\bfA +{\calE}^{(k)}_{k+1})\mathcal{R}(\bfV_k)$, in the $\bfR_{k+1}^{-1}$-weighted inner product induced. These expressions clearly reflect that the obliqueness of the projections comes only from the inexactness (as they would be orthogonal if $\calEk=\bfzero$, and would coincide with the ones associated to LSQR applied to the minimization of $g^{(k)}(\bfx)$ in \eqref{eq: gk_exact}). 

A potential shortcoming of all these methods is that, further elaborating on \eqref{eq:projrhs}, the right-hand-side vector in formulations \eqref{iCGLS} and \eqref{eq:projected_classic}, and all the mathematically equivalent versions of them, reads
\[
\beta t_{1,1}\bfe_1=\beta\bfV_k\t(\bfA+\calEk)\t\bfR_{k+1}^{-1}\bfu_1=\beta\bfV_k\t\bfA\t\bfR_{1}^{-1}\bfu_1=\bfV_k\t\bfA\t\bfR_{1}^{-1}\bfr_0,
\]
so that the right-hand-side vector appearing in the projected inexact normal equations \eqref{eq:projected_classic} is effectively always computed using the initial $\bfR_1^{-1}$ and may be quite different from the vector $\bfV_k\t\bfA\t\bfR_{k+1}^{-1}\bfr_0$ that appears in the projected normal equations \eqref{eq: projNEgk} associated to the iteration-dependent exact functional $g^{(k)}(\bfx)$. 

\begin{remark}\label{rem:DAvsAD}
Note that $\widehat{\nabla g^{(k)}}(\bfx)$ cannot be regarded as the exact gradient of a certain functional $\widehat{ g}^{(k)}(\bfx)$ somewhat related to $g^{(k)}(\bfx)$, as the matrix $(\bfA + {\calE}^{(k)}_{k+1})\t \bfR_{k+1}^{-1} \bfA$ appearing in its first term is not symmetric in general; this is the case only if $(\calEk)\t\bfR_{k+1}^{-1}\bfA$ is symmetric. For this fundamental reason, performing approximation before differentiation leads to a different class of flexible/inexact solver that is described in the next subsection.
\end{remark}

\begin{remark}\label{rem: iLSMR} {The main focus of this paper is on minimal residual methods, where at least the classical methods described in \Cref{sect:PCGLS_PLSQR} minimize the (preconditioned) residual norm when the solution is restricted to a given subspace. However, the DAP and the other mathematically equivalently methods derived so far in this subsection can also be used to define an inexact LSMR method where, at the $k$th iteration, one determines the solution $\bfx_k$ by imposing the condition \eqref{def:xsubspinex} (to restrict $\bfx_k$ to the inexact Krylov approximation subspace) and by replacing the orthogonality condition \eqref{eq:NE1_Galerkin_inexact} by 
\begin{equation}\label{eq:LSMR_inexact}
\widehat{\nabla g^{(k)}}(\bfx)=(\bfA + {\calE}^{(k)}_{k+1})\t \bfR_{k+1}^{-1} (\bfA\bfx-\bfr_0)\perp (\bfA+\calE_{k+1}^{(k)})\t\Rki\bfA\mathcal{R}(\bfV_{k})\,.
\end{equation}
Using the new inexact GK factorization \eqref{eq:iGK}, the above conditions lead to
\[
\bfx_k=\bfV_k\bfs_k,\quad\mbox{where}\quad(\bfT_{k+1}\bfM_k)\t(\bfT_{k+1}\bfM_k)\bfs_k =\beta t_{1,1}(\bfT_{k+1}\bfM_k)\t\bfe_1\,.
\]
We refer to the method so defined as DAP-LSMR. In the framework of \cite{Eiermann_Ernst_2001}, DAP-LSMR determines $(\bfA+\calEk)\t\Rki\bfA\bfx_k\in(\bfA+\calEk)\t\Rki\bfA\calR(\bfV_k)$ 
as an orthogonal projection of $(\bfA+\calEk)\t\Rki\bfr_0$ onto $(\bfA+\calEk)\t\Rki\bfA\calR(\bfV_k)$, and therefore can be regarded as a `minimal $\widehat{\nabla g^{(k)}}$' (M-$\widehat{{\nabla}g^{(k)}}$) method.}
\end{remark}

\begin{remark}
The DAP, DPA and PDA methods do not enjoy any clear optimality property (i.e., in terms of minimization of inexact normal equation residuals $\widehat{\nabla g^{(k)}}(\bfx)$ in the approximation subspace $\calR(\bfV_k)$ for the solution). This was already observed in \cite{Gazzola2021inexact} for the iCGLS method, which is mathematically equivalent to the current DAP, DPA and PDA methods. Because of the oblique projection of $(\bfA + {\calE}^{(k)}_{k+1})\t \bfR_{k+1}^{-1}\bfr_0$ happening in \eqref{eq:NE1_Galerkin_inexact}, such methods can be regarded as inexact FOM versions, applied to the inexact normal equations. As explained in \Cref{rem: iLSMR}, the DAP-LSMR method minimizes the norm of the inexact normal equation residuals $\widehat{\nabla g^{(k)}}(\bfx)$ over the same subspace as the DAP, DPA and PDA methods and, therefore, may be regarded as an inexact GMRES applied to the inexact normal equations. For this reason, one may exploit known results about relations between the FOM and GMRES residual (norms) to assess how far the DAP, DPA and PDA inexact normal residual norms are from being optimal; see, e.g., \cite{FOM_GMRES}.
\end{remark}

\paragraph{Inexactness estimates} We conclude this section by providing inexactness estimates for the inexact normal equation residuals computed by the DAP, DPA and PDA methods, i.e., given an approximation $\bfx_k\in\calR(\bfV_k)$, we bound the difference between the exact $\nabla g^{(k)}(\bfx_k)$ and the inexact $\widehat{\nabla g^{(k)}}(\bfx_k)$, both appearing in \eqref{eq:inexact_gradient}. To do so, we adapt the bounds already provided in \cite{Gazzola2021inexact, Simoncini2003Inexact}. Taking $\bfE_i^{(k)}$ as in \eqref{eq:errors1}, 
\begin{eqnarray*}
\nabla g^{(k)}(\bfx_k)-\widehat{\nabla g^{(k)}}(\bfx_k)&=& 
-(\calEk)\t\Rki\bfA\bfV_k\bfs_k + (\calEk)\t\Rki\bfr_0\\
 &=&-\bfA\t[\bfE_1^{(k)}\bfu_1,\dots,\bfE_{k+1}^{(k)}\bfu_{k+1}]{\bfM}_k\bfs_k+\bfA\t\bfE_1^{(k)}\bfr_0
\end{eqnarray*}
we straightforwardly derive the bound
\begin{equation}\label{eq:bound1}
\|\nabla g^{(k)}(\bfx_k)-
\widehat{\nabla g^{(k)}}(\bfx_k) \|\leq
\|\bfA\t\bfE_1^{(k)}\bfr_0\|+\sum_{i=1}^{k+1}\|\bfA\t\bfE_i^{(k)}\||[\bar{\bfs}_k]_i|\,,
\end{equation}
where $\bar{\bfs}_k={\bfM}_k\bfs_k$.
Note that the theoretical bound in \eqref{eq:bound1} can be very expensive to compute: at the $k$th iteration it requires $k$ additional matrix-vector products with $\bfA\t$. Therefore, in practice, we can also consider computing the following cheaper but looser bound
\[
\|\nabla g^{(k)}(\bfx_k)-
\widehat{\nabla g^{(k)}}(\bfx_k)\|\leq
\|\bfA\t\|\|\bfE_1^{(k)}\bfr_0\|+\sum_{i=1}^{k+1}\|\bfA\t\|\|\bfE_i^{(k)}\||[\bar{\bfs}_k]_i|\,.
\]
Here $\|\bfA\t\|$ can be estimated by running only a few (standard) Golub-Kahan bidiagonalization iterations with $\bfA$, which could be also used to find an initial approximation $\bfx_0$ of $\bar{\bfx}^\ast$ in \eqref{eq: LSR}, \eqref{eq: LSp}. 

\begin{remark}\label{rem: rest1} 
\msl{Inexactness estimates are important for understanding how far the original problem is from the one being solved in practice. When the mismatch is too great, this can be rectified by means of a restart. In practice, this can have a great impact in the quality of the reconstructed solution. For this type of classical inexact methods, this is discussed in depth  \cite{Gazzola2021inexact}, where different restarting strategies are also specified.} Note that the mismatch for the DAP, DPA and PDA methods is always measured in terms of the normal equations residual norm.
\end{remark}

\subsubsection{APD, ADP and PAD solvers, i.e., new flexible Krylov solvers}\label{sect: inexact2}
In this subsection, we propose a new flexible/inexact Krylov framework, where the exact objective function $g(\bfx)$ in \eqref{eq: g} (or the exact functional at the $k$th iteration $g^{(k)}(\bfx)$ in \eqref{eq: gk_exact}) is first approximated \msl{(A)} by an iteration-dependent convex functional $\bar{g}^{(k)}(\bfx)$, and then differentiated \msl{(D)}. As in the previous subsection, projection (P) can either happen before or after differentiation, leading to LSQR-like and CGLS-like solvers, respectively, which are again mathematically equivalent (in the current framework). With respect to the methods presented in the previous subsection, the methods derived in the present subsection have the advantage of working with explicit optimality conditions for $\bar{g}^{(k)}(\bfx)$, \msl{meaning that we can characterize the solution at each iteration as the minimizer of $\bar{g}^{(k)}(\bfx)$ restricted to a solution space whose dimension increases with the iterations. 
For some constructions of $\bar{g}^{(k)}(\bfx)$, this allows us to theoretically control the difference between the value of the iteration-dependent objective function at each iteration and a desired exact function $f(\bfx)$; see Proposition \ref{lemma1}. Note that, in contrast, the traditional inexact methods in \Cref{sect: inexact1} can only query inexact gradient norms, i.e., inexact normal equations residual norms.} 

We start by defining the following \textit{inexact objective function at the $k$th iteration} 
\begin{equation}\label{def:itdep_obj}
\begin{split}
    \bar g^{(k)}(\bfx) :=& \nicefrac{1}{2}\left(\bfx\t\bfA\t\bar{\bfR}_{k+1}\bfA\bfx \!-\! \bfx\t\bfA\t\bar{\bfR}_{k+1}\bfr_0 \!-\!\bfx\t\bfA\t\bar{\bfR}_{k+1}\t\bfr_0 
+\bfr_0\t\bar{\bfR}_{k+1}\bfr_0\right)\\
=&\nicefrac{1}{2}\left(\bfx\t\bfA\t\bar{\bfR}_{k+1}^{(S)}\bfA\bfx - 2\bfx\t\bfA\t\bar{\bfR}_{k+1}^{(S)}\bfr_0
+\bfr_0\t\bar{\bfR}_{k+1}^{(S)}\bfr_0\right)\\
=&\:\nicefrac{1}{2}\left\|\bfA\bfx -\bfr_0\right\|_{\bar{\bfR}_{k+1}^{(S)}}^2,
\end{split}
\end{equation}
where the matrix $\bar\bfR_{k+1}$ is defined as in \eqref{defRkbar} and $\bar\bfR_{k+1}^{(S)}=\nicefrac{1}{2}(\bar{\bfR}_{k+1}+\bar{\bfR}\t_{k+1})$ is its symmetric part.
 
Compared to the function $g^{(k)}(\bfx)$ in \eqref{eq: gk_exact}, 
the weights in $\bar{g}^{(k)}(\bfx)$ are defined with respect to the current and all the previous iteration-dependent weights $\bfR_{i}^{-1}$, $i=1,\dots,k+1$. Moreover, the rank of $\bRSk$ is at most $2k+2$ so that, unless $k$ is at least equal to the smallest integer greater or equal to $\nicefrac{m-2}{2}$ (which is undesirable in large-scale settings), $\bRSk$ is singular and $\left\|\cdot\right\|_{\bar{\bfR}_{k+1}^{(S)}}$ is a semi-norm. However we will see that, for all the solvers in this subsection, $\bar{g}^{(k)}(\bfx)$ is only evaluated for $\bfx\in\calR(\bfV_k)$ (the approximation subspace for the solution, defined in \eqref{def: RV_k} with $\hat{\bfu}_1=\bfr_0$), implying that $\left\|\cdot\right\|_{\bar{\bfR}_{k+1}^{(S)}}$ is only applied to vectors belonging to $\calR(\bfU_{k+1})$, and therefore it is a norm. Even if the approximation subspaces for the $k$th solution of all the flexible/inexact solvers considered in this and the previous sections coincide, the two classes of solvers are obviously different, as the initial problem formulations and the respective projected problems are different, leading to different solution coefficients (i.e., solutions of the projected problems) and, eventually, different solutions. This also agrees with Remark \ref{rem:DAvsAD}. In Section \ref{sect: NumExp} we illustrate that, in many circumstances, the new solvers presented in this subsection deliver superior performance with respect to traditional solvers (considered in the previous subsection).

We first consider the approximation (A) in \eqref{def:itdep_obj}, and we then perform a projection (P), analogous to the one underlying the M-$g$ method in Section \ref{sect:PCGLS_PLSQR}: namely, we constrain the variable in \eqref{def:itdep_obj} to belong to the space $\calR(\bfV_k)$, i.e., we compute 
\begin{equation}\label{def:xsubspflex}
\bfx_k=\bfV_k\bfs_k,\quad\mbox{with}\quad\calR(\bfV_k)=\calK_k(\bfA\t\bar\bfR_{k+1}\bfA,\bfA\t\bar\bfR_{k+1}\bfr_0)
\end{equation}
by solving
\begin{equation}\label{eq:skminbargk}
\bfs_k=\argmin_{\bfs\in\bbR^k}\bar{g}^{(k)}_k(\bfs),
\end{equation}
where, using the FGK relations \eqref{eq:FGK_2},
\begin{eqnarray}\label{eq: gkbars}
    \bar g^{(k)}_k(\bfs) &=& \nicefrac{1}{2}\left(\bfs\t\bfV_k\t\bfA\t\bar{\bfR}_{k+1}^{(S)}\bfA\bfV_k\bfs - 
2\bfs\t\bfV_k\t\bfA\t\bar{\bfR}_{k+1}^{(S)}\bfr_0 
+\bfr_0\t\bar{\bfR}_{k+1}^{(S)}\bfr_0\right).
\end{eqnarray}
To solve the projected problem \eqref{eq:skminbargk} we differentiate (D) $\bar g_k^{(k)}(\bfs)$, getting the equation
\[
\nabla \bar g_k^{(k)}(\bfs)=\bfV_k\t\bfA\t\bar\bfR_{k+1}^{(S)}\bfA\bfV_k \bfs - \bfV_k\t\bfA\t\bar{\bfR}_{k+1}^{(S)}\bfr_0=\bfzero,
\]
where $\bar\bfR_{k+1}^{(S)}$ is the symmetric part of $\bar\bfR_{k+1}$. 
Considering again the FGK relations in \eqref{eq:FGK_2}, and performing some straightforward algebraic manipulations (similar to the ones used to derive problem \eqref{eq:projected_classic}), 
we eventually recover $\bfs_k$ as the vector satisfying
\begin{equation}\label{eq: NE2}
({\bfT}_{k,k+1}{\bfM}_k+ ({\bfT}_{k,k+1}{\bfM}_k)\t)\bfs_k = \beta t_{1,1}\bfe_1 + \bfM_k\t\bfY_{k+1}\t\bfr_0,
\end{equation}
where, like in the previous subsection, $\beta=\|\bfr_0\|_{2}$, $t_{1,1}=[\bfT_{k+1}]_{1,1}$, and 
${\bfT}_{k,k+1}$
is ${\bfT}_{k+1}$ without its last row. The APD method so derived is a LSQR-like solver. In the framework of \cite{Eiermann_Ernst_2001}, this method may be regarded as a `minimal inexact objective function' (M-$\bar g^{(k)}$) method that determines $\bfA\t\bar\bfR_{k+1}^{(S)}\bfA\bfx_k$ by performing an oblique projection of $\bfA\t\bar\bfR_{k+1}^{(S)}\bfr_0$ \msl{onto $\bfA\t\bar\bfR_{k+1}^{(S)}\bfA\calK_k(\bfA\t\bar\bfR_{k+1}\bfA,\bfA\t\bar\bfR_{k+1}\bfr_0)$}, orthogonal to $\calK_k(\bfA\t\bar\bfR_{k+1}\bfA,\bfA\t\bar\bfR_{k+1}\bfr_0)$. Equivalently, it is based on a projection procedure
that determines $(\bRSk)^{\nicefrac{1}{2}}\bfA\bfx_k$ by performing an orthogonal projection of $(\bRSk)^{\nicefrac{1}{2}}\bfr_0$ onto \linebreak[4]$(\bRSk)^{\nicefrac{1}{2}}\bfA\calK_k(\bfA\t\bar\bfR_{k+1}\bfA,\bfA\t\bar\bfR_{k+1}\bfr_0)$. 
%

Similarly, one can define an equivalent ADP method as follows. The starting point is the iteration-dependent approximate (A) functional $\bar{g}^{(k)}(\bfx)$ in \eqref{def:itdep_obj}, which is then  differentiated (D) to obtain 
\[
\nabla \bar g^{(k)}(\bfx)=\bfA\t\bRSk 
\bfA \bfx -  \bfA\t\bar{\bfR}_{k+1}^{(S)}\bfr_0
\]
and eventually projected (P) to get
$\bfx_k=\bfV_k\bfs_k$, with $\bfs_k\in\bbR^k$, such that 
\[
\nabla \bar{g}^{(k)}(\bfx_k)\perp \calR(\bfV_k)=\calK_k(\bfA\t\bRk\bfA,\bfA\t\bRk\bfr_0)\quad\Longleftrightarrow\quad\mbox{$\bfs_k$ solves \eqref{eq: NE2}}.  
\]
The ADP method so derived is a CGLS-like solver. In the framework of \cite{Eiermann_Ernst_2001}, this method may be regarded as an `orthogonal inexact normal equation residual' (O-$\nabla\bar{g}^{(k)}$) that determines $\bfA\t\!\bRSk\!\bfA\bfx_k$ by performing an oblique projection of $\bfA\t\!\bRSk\!\bfr_0$ onto $\bfA\t\bRSk\bfA\calR(\bfV_k)$, 
orthogonal to $\calR(\bfV_k)$, exactly as the LSQR-like APD solver. Remarkably, in the flexible/inexact setting, and after approximating the exact $g(\bfx)$ in \eqref{eq: g} (or $g^{(k)}(\bfx)$ in \eqref{eq: gk_exact}) by $\bar{g}^{(k)}(\bfx)$ in \eqref{def:itdep_obj}, APD and ADP are the only LSQR-like and CGLS-like solvers that are mathematically equivalent when applied to minimize $\bar{g}^{(k)}(\bfx)$ and solve $\nabla \bar{g}^{(k)}(\bfx)=\bfzero$.

Finally, and again equivalently, one can start by projecting (P) the original $g(\bfx)$ to get an approximation $\bfx_k$ still belonging to the subspace $\calR(\bfV_k)$ as in \eqref{def:xsubspflex}, i.e., one computes
\[
\bfx_k=\bfV_k\bfs_k,\quad\mbox{with}\quad\bfs_k=\argmin_{\bfs\in\bbR^k}g_k(\bfs)\quad\mbox{and}\quad\mbox{
$g_k(\bfs)$ defined as in \eqref{eq:NE1_PetrovGalerkin}},
\]
and then takes $\bar{g}_k^{(k)}(\bfs)$ in \eqref{eq: gkbars} to be an approximation (A) of $g_k(\bfs)$, eventually obtaining the same formulation as \eqref{eq:skminbargk}. By applying differentiation (D) to impose the optimality conditions, one obtains again that $\bfs_k$ solves problem \eqref{eq: NE2}, resulting in a LSQR-like PAD method equivalent to the first APD method of this class. Moreover, both PAD and APD share the same description in terms of the formalism in \cite{Eiermann_Ernst_2001}, including their characterization as M-$\bar{g}^{(k)}$ methods. 

\paragraph{Inexactness estimates} Similarly to what we did in \Cref{sect: inexact1}, we provide some inexactness estimates, now concerning the gap between the exact functional $g^{(k)}(\bfx)$ at iteration $k$, defined in \eqref{eq: gk_exact}, and the its inexact version $\bar{g}^{(k)}(\bfx)$, defined in \eqref{def:itdep_obj}, both evaluated at $\bfx_k\in\calR(\bfV_k)$. 

Assume that, at iteration $k$ of our new method, we have an approximate solution ${\bar\bfx_k}=\bfx_0+\bfV_k\bfs_k$
First note that, similarly to what already stated for $\bar{g}^{(k)}$, when restricting the exact objective function at the $k$th iteration $g^{(k)}$ to $\calR(\bfV_k)$, i.e., when computing 
\begin{equation}\label{eq:res_norm}
    g^{(k)}_k(\bfs)= \nicefrac{1}{2} \|\bfA\bfV_k\bfs-\bfr_0\|_{\bfR_{k+1}^{-1}}^2 =  \nicefrac{1}{2}  \|\bfU_{k+1} {\bfM}_k\bfs-\bfU_{k+1} \beta \bfe_1\|_{\bfR_{k+1}^{-1}}^2
\end{equation}
 on any possible coefficients $\bfs\in\bbR^k$, 
$\bfR_{k+1}^{-1}$ is only applied to vectors in $\calR(\bfU_{k+1})=\text{span}\{\bfu_1,\dots,\bfu_{k+1}\}$. 
Now, let us define the following residual as
\[  \bfr_k=\bfr_0-\bfA\bfV_k\bfs_k  = \bfU_{k+1} (\beta \bfe_1-{\bfM}_k\bfs_k) \in \mathcal{R}(\bfU_{k+1}),\]
and note that, for any $\bfu=\sum_{i=1}^{k+1}(\bfu_i\t\bfu)\bfu_i\in\mathcal{R}(\bfU_{k+1})$,
\begin{eqnarray*}
        \bfu\t (\bar{\bfR}_{k+1}-\bfR_{k+1}^{-1}) \bfu= \sum_{i=1}^{k+1}  \bfu\t (\bfR^{-1}_i -\bfR^{-1}_{k+1}) \bfu_i \bfu_i\t \bfu.
\end{eqnarray*}
Then, the magnitude of the difference between the exact and the inexact projected functionals on the solution at any iteration $k$ can be bounded as follows: 
\begin{eqnarray*}
        |\bar g^{(k)}_k(\bfs_k)- g^{(k)}_k(\bfs_k)| &=& \nicefrac{1}{2}\left|\sum_{i=1}^{k+1}  \bfr_k\t (\bfR^{-1}_i -\bfR^{-1}_{k+1}) \bfu_i \bfu_i\t \bfr_k\right| \\
        &\leq& \nicefrac{1}{2}\sum_{i=1}^{k+1} \left| \bfr_k\t (\bfR^{-1}_i -\bfR^{-1}_{k+1}) \bfu_i \bfu_i\t \bfr_k\right| \\
        &\leq& \nicefrac{1}{2}\| \bfr_k \|  \sum_{i=1}^{k+1} \| \bfR^{-1}_i -\bfR^{-1}_{k+1} \| 
        |[{\bfM}_k\bfs_k-\beta \bfe_1]_i|
\end{eqnarray*}        
where we have used that $\bfu^T \bar \bfR_{k+1}^{(S)}\bfu = \bfu^T \bar \bfR_{k+1} \bfu $.

\begin{remark}\label{rem: rest2} Similarly to the classical inexact methods, restarting when the original problem is too far from the problem that is solved in practice, is crucial to obtain good approximations of the solution.
\end{remark}
\begin{remark}
As mentioned in the introduction in Section \ref{sec: Intro}, a canonical example for our framework is the approximation of the following non-linear problem
\begin{eqnarray}\label{eq:MM_f}
f(\bfx)= \nicefrac{1}{2} \|\bfA\bfx-\bfb\|^2_{\bfR^{-1}(\bfx)}
\end{eqnarray}
arising, for example, from considering the $\ell_p$ optimization in \eqref{eq: LSp} where $\bfR^{-1}(\bfx)$ is defined in \eqref{eq:weights_lp}. In this case, the inexactness estimates can be useful to understand the convergence of the flexible MM schemes, as detailed in the following proposition. Theoretically, one should restart once monotonicity of the objective function throughout the iterates cannot be guaranteed. Note that, even if the exact conditions are expensive to check in practice, this theory underpins the effect of restarting in the new flexible solvers.
\end{remark}
\begin{proposition}\label{lemma1}Monotonicity properties for flexible MM schemes. 
Consider the approximate solutions given the new flexible solvers to be $\{\bfx_k\}_{k=1,...}$, and assume that $\bfR^{-1}_{k+1} = \bfR^{-1}(\bfx_{k-1})$.  Then, for $f(\bfx)$ defined in \eqref{eq:MM_f}, we can guarantee monotonicity in the objective value at each iteration $k$ if sufficient progress is achieved in the corresponding inexact approximations $\bar{g}^{(k)}(\bfx)$. Mathematically, 
\[
f(\bfx_k) < f(\bfx_{k-1}) 
\quad
\Leftarrow
\quad
\bar{g}^{(k)}(\bfx_{k-1}) - \bar{g}^{(k)}(\bfx_{k}) >
| \bfnu\t \sum_{i=1}^{k+1} (\bfR_i^{-1} -\bfR_{k+1}^{-1} ) \bfu_i\bfu_i\t \bfnu 
|, 
\]
for $\bfnu = \bfr_k-\bfr_{k-1} = \bfU_k \left( \bfy_k - \begin{bmatrix}
    \bfy_{k-1} \\ 0
\end{bmatrix}\right)$.
\end{proposition} 
\begin{proof}
Recall the definition of $g^{(k)}(\bfx)$ for this specific case,
\[
g^{(k)}(\bfx)= \nicefrac{1}{2}\|\bfA\bfx-\bfb\|^2_{\bfR_{k+1}^{-1}}= \nicefrac{1}{2}\|\bfA\bfx-\bfb\|^2_{\bfR^{-1}(\bfx_{k-1})},
\]
and consider the following chain of in/equalities:
\[
f(\bfx_k) < g^{(k)}(\bfx_k) \leq g^{(k)}(\bfx_{k-1}) = f(\bfx_{k-1}).
\]
The first inequality is always true since $g^{(k)}(\bfx)$ is a majorant of $f(\bfx)$, and the last equality is always true since  $g^{(k)}(\bfx)$ is tangent to $f(\bfx)$ at $\bfx_{k-1}$. Therefore, we only need to find explicit conditions for the central inequality to hold. We start by writing
\begin{eqnarray*}
g^{(k)}(\bfx) = \bar g^{(k)}(\bfx)- (\bar g^{(k)}(\bfx)-g^{(k)}(\bfx)),
\end{eqnarray*}
and we use similar derivations to the ones use for the {inexactness estimates} earlier in this section to express the second term. 

In particular, we use that for any $\bfu\in\mathcal{R}(\bfU_{k+1})$
\[\bfR_{k+1}^{-1} \sum_{i=1}^{k+1} \bfu_i\bfu_i\t \bfu = \bfR_{k+1}^{-1} \bfu, \]
so that, since  $\bfr\in\mathcal{R}(\bfU_{k+1})$,
\begin{eqnarray*}
&&g^{(k)}(\bfx_{k-1})-g^{(k)}(\bfx_{k}) = \\
&&=\bar{g}^{(k)}(\bfx_{k-1}) - \bar{g}^{(k)}(\bfx_{k}) - (\bfr_{k-1}-\bfr_{k})\t \sum_{i=1}^{k+1} (\bfR_i^{-1} -\bfR_{k+1}^{-1})\bfu_i\bfu_i\t (\bfr_{k-1}-\bfr_{k}) 
\end{eqnarray*}
where we know that 
\[\bar{g}^{(k)}(\bfx_{k-1}) - \bar{g}^{(k)}(\bfx_{k}) = K> 0 \]
since $\bfx_{k-1}\in \mathcal{R}(\bfV_{k-1}) \subset \mathcal{R}(\bfV_{k})$ and $\bfx_{k}$ minimizes $\bar{g}^{(k)}(\bfx)$ in $\mathcal{R}(\bfV_{k})$. Therefore, we can guarantee that the inequality will hold if the magnitude of the error term is smaller than $K$ i.e.;
\[
|\left( \bfr_k - \bfr_{k-1} \right)\t  \, \sum_{i=1}^{k-1} \bfE_i^{-1}\bfu_i\bfu_i\t \left( \bfr_k - \bfr_{k-1} \right)| < K.
\]

Interestingly, this is always true in the first iteration after each restart, since there isn't an error in the first term. 
\end{proof}
\begin{remark}
Note that this analysis is only possible for the new APD, PAD and ADP solvers; because the inexactness estimates for classical solvers are only with respect to the gradient of the objective function and not the objective function itself.
\end{remark}

\section{Numerical experiments}\label{sect: NumExp}
This section presents two numerical experiments designed to evaluate the performance of both the classical and the new inexact methods developed in this paper. In both problems, the measurements are corrupted with (sparse) salt-and-pepper noise, so it is reasonable to use an $\ell_1$ data-fitting term.  These examples highlight both the efficiency of the new methods with respect to other algorithms, as well as the improvement in reconstruction quality achieved by using the appropriate minimization problem. 

Note that, as explained in Remark \ref{rem: rest1} and Remark \ref{rem: rest2}, inexact methods greatly improve their performance if equipped with a restarting strategy. For these examples, we restart at the first iteration $k$ such that
\small 
\begin{equation}\label{eq:rest1}
\exists \,i,j \in\{1,\dots,k+1\}: \, [\bf\eps_i]_j-[\bf\eps_i]_{j-1}>0 \quad \text{for} \quad \bf\eps_i=\max(|\diag{\bfR_i^{-1}-\bfR_{k+2}^{-1}}\hspace{-2.5pt}|\,)\,.
\end{equation}
\normalsize
where $\bfR_{k+2}^{-1}=\bfR_{k+2}^{-1}=\bfR^{-1}(\bfx_k)$, or 
\small 
\begin{equation}\label{eq:rest2}
\frac{\|\bfA\bfx_k-\bfb\|_2-\|\bfA\bfx_0-\bfb\|_2}{\|\bfA\bfx_k-\bfb\|_2} > tol.
\end{equation}
\normalsize

Throughout the section, in order to evaluate the performance of the new algorithms and compare them to other standard methods, we report the relative error norm at iteration $k$, defined as $\|\bfx_k-\bfx\true\|_2/\|\bfx\true\|_2.$

All computations were carried out in MATLAB R2023a, using both solvers and test problems available within the IR Tools \cite{IRtools} and AIR Tools II \cite{hansen2018air} MATLAB toolboxes.

\paragraph{Example 1: satellite image deblurring problem}
This example corresponds to an image deblurring problem where both the exact image and the available data (affected by Gaussian blur and salt-and-pepper noise) have $256\times 256$ pixels and are displayed in Figure \ref{fig:Ex1_data}, together with the Gaussian point spread function defining the blur and noise array (also regarded as residual associated to the exact solution $\bfb-\bfA\bfx_{\true}=:\bfr_{\true}$). In particular, the salt-and-pepper noise is such that about 10\% of the pixels are randomly set to values 0 or 1. Since the noise in the measurements is sparse, it is beneficial to use an $\ell_1$ fit-to-data term to obtain a solution of better quality. 

\begin{figure}[h!]
    \centering
    \includegraphics[width=\linewidth]{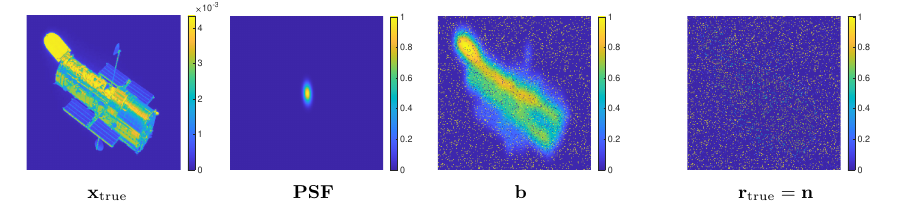}
    \caption{Data for the first example, representing an image deblurring test problem. From left to right: true solution, point spread function for the blur, noisy blurred image, and noise. The latter is displayed using the square root of the absolute values of each element.
    }
    \label{fig:Ex1_data}
\end{figure}

For this example, we compare the performance of all the inexact/flexible methods considered in this paper. Namely, the new APD method, 
and the other methods operating in the same framework: DAP (corresponding to classical inexact CGLS and inexact LSQR)  and DAP-LSMR (corresponding to classical inexact LSMR).

Since all of these are inexact methods, we compare their performance with and without restarts, where the restarts are triggered using the condition \eqref{eq:rest1} and using $tol=0.1$. Moreover, we also compare the new method to the classic LSQR, which considers a standard LS problem (with $\ell_2$-norm), and a standard IRLS solver for an $\ell_1$ fit-to-data term, using LSQR as an inner solver. Last, we also compare preconditioned LSQR using the weights evaluated in the exact solution – this is of course not a valid method, since it requires the knowledge of the exact solution, and it is only displayed as a lower bound for the error. 

The relative error histories for the different solvers are displayed in Figure \ref{fig:Ex1_RREs}. One can observe that the new method, jointly with the classic DAP and DAP-LSMR, produces the best reconstructions, as well as the fastest convergence behaviour, jointly with DAP.

\begin{figure}[ht!]
    \centering
    \includegraphics[width=\textwidth]{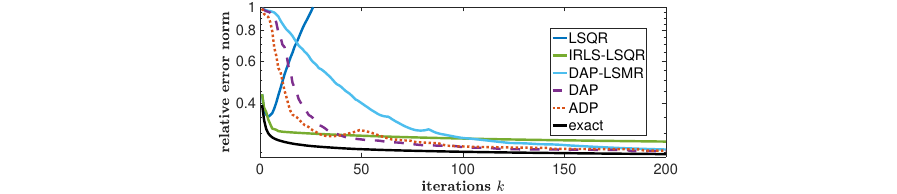}
    \includegraphics[width=\textwidth]{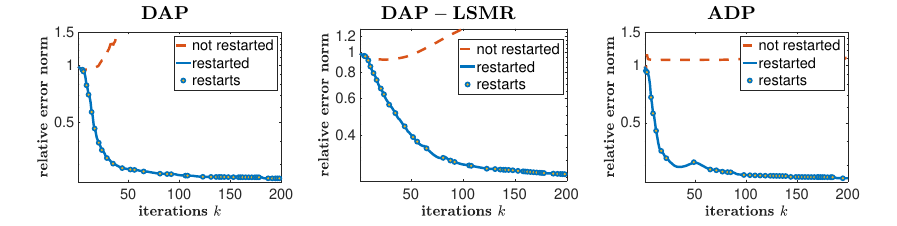}
    \caption{Example 1. Relative error norm values versus iteration number for a variety of solvers, with specific emphasis on the effect of restarting (second row). Note that the label `exact' corresponds to preconditioned LSQR using the weights evaluated in the exact solution.}
    \label{fig:Ex1_RREs}
\end{figure}

The quality of the reconstructions is illustrated in the first row of Figure~\ref{fig:Ex2_recs}, which shows the best reconstruction obtained by each method (excluding LSQR with weights evaluated in the exact solution used as a preconditioner). The second row of  Figure~\ref{fig:Ex2_recs} presents the corresponding error images, obtained by reshaping the reconstruction errors. Figure~\ref{fig:Ex2_relrecs} further displays the residuals and residual errors, also reshaped as images. For visualization purposes, the errors, residuals, and residual errors are shown as the absolute values of the square roots of the pixel values. Last, the basis vectors corresponding to the compared methods can be found in Figures \ref{fig:Ex1_basisvects_extra} and \ref{fig:Ex1_basisvects_extra_2} in Appendix \ref{sec:appendix1}.

\begin{figure}[ht!]
    \centering
    \includegraphics[width=\linewidth]{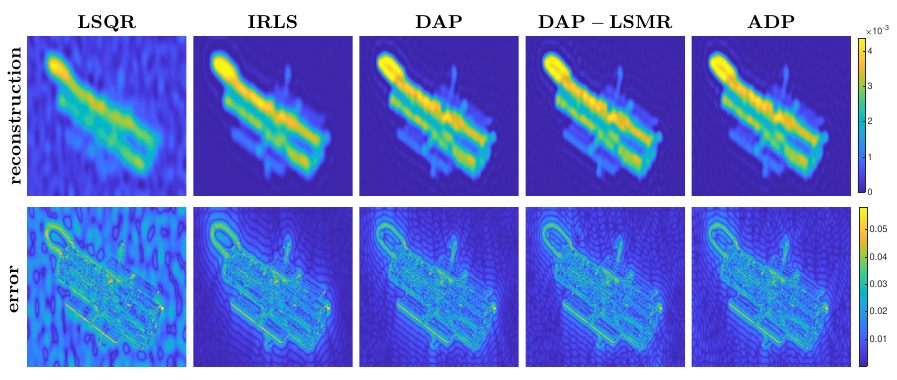} 
    \caption{Example 1. Top row: Best reconstructions for a variety of algorithms. Bottom row: Square root of the absolute values of the error vectors (all operations performed
entry-wise). }
    \label{fig:Ex2_recs}
\end{figure}

\begin{figure}[ht!]
    \centering
    \includegraphics[width=\linewidth]{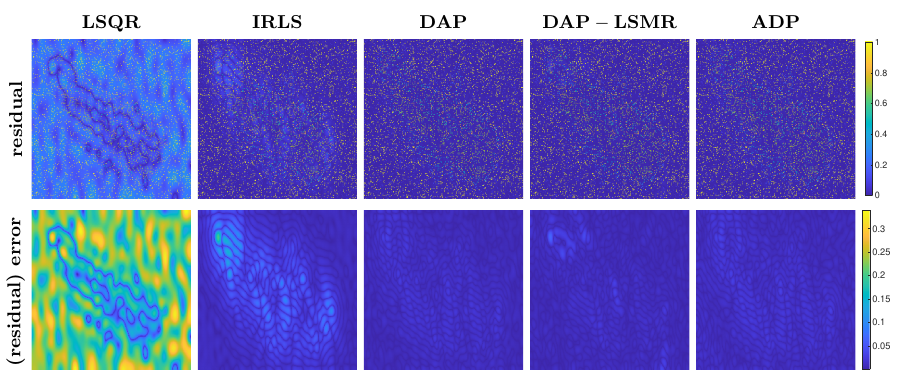} 
    \caption{Example 1. Top row: Residuals associated to the best reconstructions for a variety of algorithms. Bottom row: residual error vectors, i.e. $\bfA\bfx-\bfb-\bfr_\text{true}$. All of the images are displaying the square root of the absolute values of each of the entries. }
    \label{fig:Ex2_relrecs}
\end{figure}

\paragraph{Example 2: Tomography problem} This example corresponds to a (parallel beam) tomography test problem, where the solution has $256 \times 256$ pixels and the measurements, also know as sinogram, have $362 \times 180$ pixels (corresponding to the number of rays and the number of angles). The measurements have been normalized so that the maximum value is 1, and then `salt and pepper' noise has been added, where $10\%$ of the pixels are randomly taken to be 1 or 0. This test is created using IRtools \cite{IRtools}, and the corresponding data can be found in Figure \ref{fig:Ex3_data}.

\begin{figure}[h!]
    \centering
     \includegraphics[width=\linewidth]{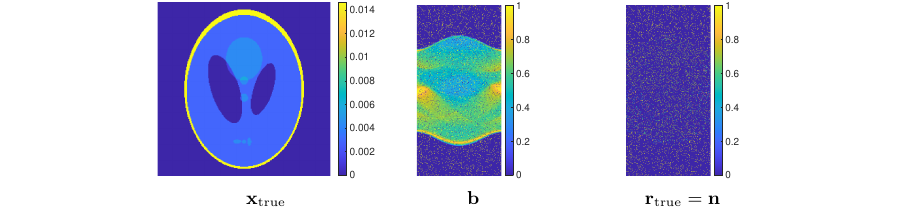}
    \caption{Data for the second example, representing a computed tomography  test problem. Left: true solution. Center: noisy measurements (a.k.a. sinogram). Right: square root of the absolute value of the noise, or the residual of the true solution (all operations performed entry-wise).}
    \label{fig:Ex3_data}
\end{figure}

The relative error norm histories for different methods can be observed in Figure~\ref{fig:Ex3_RREs}. For this example, the new inexact method (ADP), performs much better than any of the compared methods. However, from the bottom row of Figure~\ref{fig:Ex3_RREs}, it becomes evident again that restarting the iterations is necessary for the inexact methods to perform well. This matches the theory presented in Proposition~\ref{lemma1}. Note that, here, DAP and DAP-LSMR are restarted using \eqref{eq:rest1} and ADP is restarted using \eqref{eq:rest2}, and in both cases $tol=0.1$. The selection between these two criteria has been done to obtain the best results possible for each of the methods.

\begin{figure}[ht!]
    \centering
    \includegraphics[width=\textwidth]{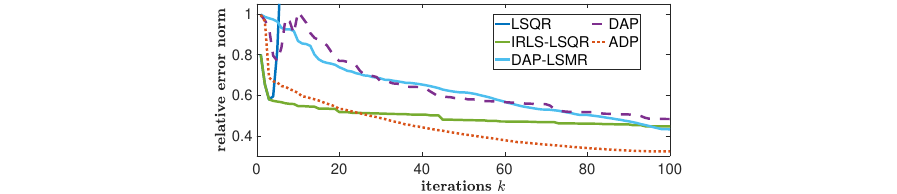}
    \includegraphics[width=\textwidth]{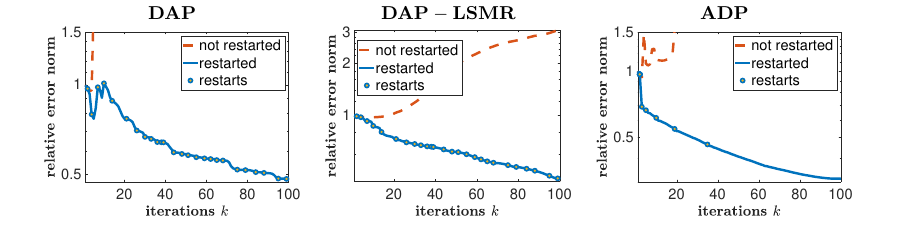}
    \caption{Example 2. Relative error norm values versus iteration number for a variety of solvers, with specific emphasis on the effect of restarting (bottom row).}
    \label{fig:Ex3_RREs}
\end{figure}

The reconstructions using different methods can be found in Figure~\ref{fig:Ex3_recs}, along with the errors (reshaped as an image). The latter have been displayed as the square root of the absolute value of each pixel, to help visualization, due to the big difference between their values. Here, one can observe again that the new ADP method produces the best quality reconstructions overall. It is also interesting to note that the error for the ADP method is less localized than that corresponding to the other methods. Moreover, the residuals, as well as the residual errors (defined as $\bfA\bfx-\bfb-\bfr_\text{true}$) can be observed in Figure \ref{fig:Ex3_relrecs}, reshaped as images. Similarly to what we can observe in Figure~\ref{fig:Ex3_recs}, the new ADP method is the one that most accurately reconstructs the true residual, or, equivalently, the noise in the measurements. Finally, the basis vectors corresponding to the compared methods can be found in Figures \ref{fig:Ex2_basisvects_extra} and \ref{fig:Ex2_basisvects_extra_2} in Appendix \ref{sec:appendix1}.

\begin{figure}[ht!]
    \centering
    \includegraphics[width=\linewidth]{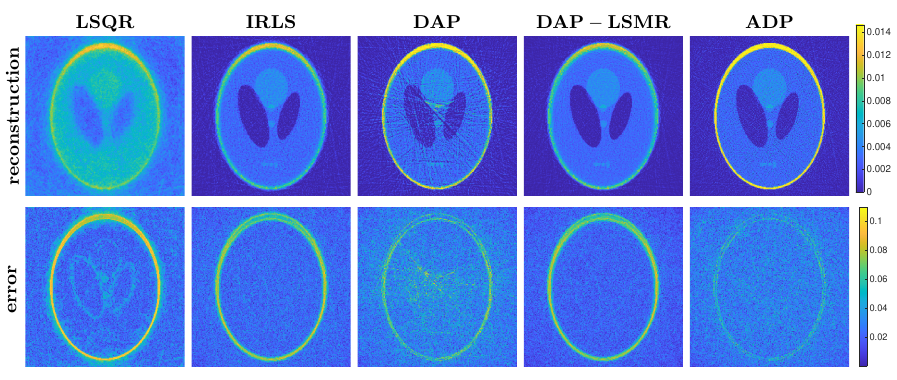} 
    \caption{Example 2. Top row: Best reconstructions for a variety of algorithms. Bottom row: Square root of the absolute values of the error vectors (all operations performed
entry-wise). }
    \label{fig:Ex3_recs}
\end{figure}

\begin{figure}[ht!]
    \centering
    \includegraphics[width=\linewidth]{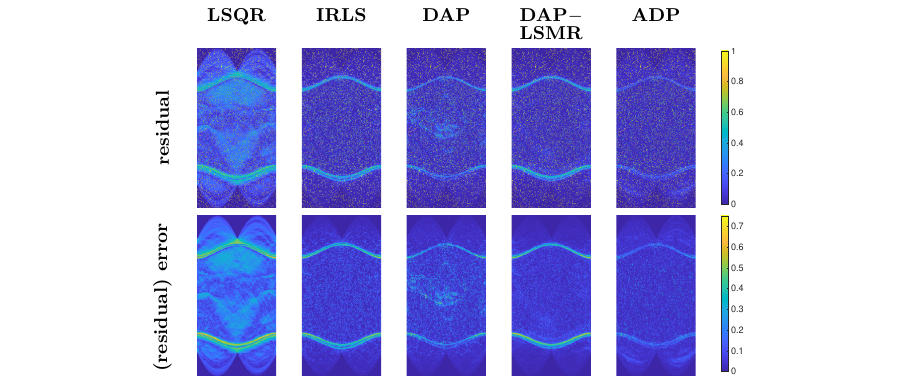} 
    \caption{Example 2. Top row: Residuals associated to the best reconstructions for a variety of algorithms. Bottom row: residual error vectors, i.e. $\bfA\bfx-\bfb-\bfr_\text{true}$. All of the images are displaying the square root of the absolute values of each of the entries. }
    \label{fig:Ex3_relrecs}
\end{figure}

\section{Conclusions and outlook} 
This paper introduces a generalized framework for inexact Krylov subspace methods, contributing an original insight on existing methods, as well as providing theoretical motivation for the development of new solvers. In particular, we present a novel inexact Golub-Kahan factorization associated to this new class of inexact Krylov methods, as well as an alternative residual constraint for the solution which can yield improved results with respect to their standard inexact counterparts. In this work, we also recall known bounds on the inexactness on traditional solvers, and we show that a more accurate measure can be given for the presented methods. This is based on the idea that the new methods are based on explicit optimality conditions of an approximated functional.  

The main application focus of this paper is a new type of flexible Krylov subspace methods, which is particularly suited for problems involving a fit-to-data functional given in a $\ell_p$ norm. The starting point for this problem is to consider a sequence of weighted least-squares problems that can be partially solved by considering iteration-dependent left preconditioning which depends on the last available approximation to the solution. The new inexact methods prove to be efficient alternatives to other methods. Moreover, for the solutions obtained with the new methods, we give conditions on the monotonicity of the original function with respect to the iterations, which motivate the use of restarts. However, these conditions are expensive to compute, so it is a topic of future work to investigate computationally cheap restarting conditions based on this theory. 

The numerical experiments illustrate both the effectiveness and efficiency of the proposed methods compared to other standard solvers for an $\ell_p$ fit-to-data term.

Last, note that both the theory as well as the new solvers can be generalized to include explicit regularization in the projection functionals. This can be either Tikhonov regularization, to obtain a so called hybrid method, or more advanced regularization functionals such as those appearing in $\ell_p-\ell_q$ regularization. However, we leave this as a topic of further research. 

\bibliographystyle{plain}
\bibliography{SIRev_bibliography}

\appendix

\section{Basis vectors}\label{sec:appendix1}
In this Appendix, we show a selection of basis vectors constructed by the different methods compared in Section \ref{sect: NumExp}. Note that, even if some of these processes share the same flexible Golub-Kahan factorization described in Algorithm \ref{alg: FGK}, the iteration-dependent preconditioning depends on the approximations of the solution at each iteration and, therefore, the solution subspace depends indirectly on the optimality conditions that differentiate the methods. Note that, particularly for the computed tomography example, one can clear observe the difference between the two sets of basis vectors. Moreover, it is interesting to note that the basis vectors look very different for the different methods. Understanding how the properties of the basis vectors constructed by the ADP method relate to the improved quality of their results is left as future work.

\begin{figure}[ht]
\centering
\includegraphics[width=\textwidth]{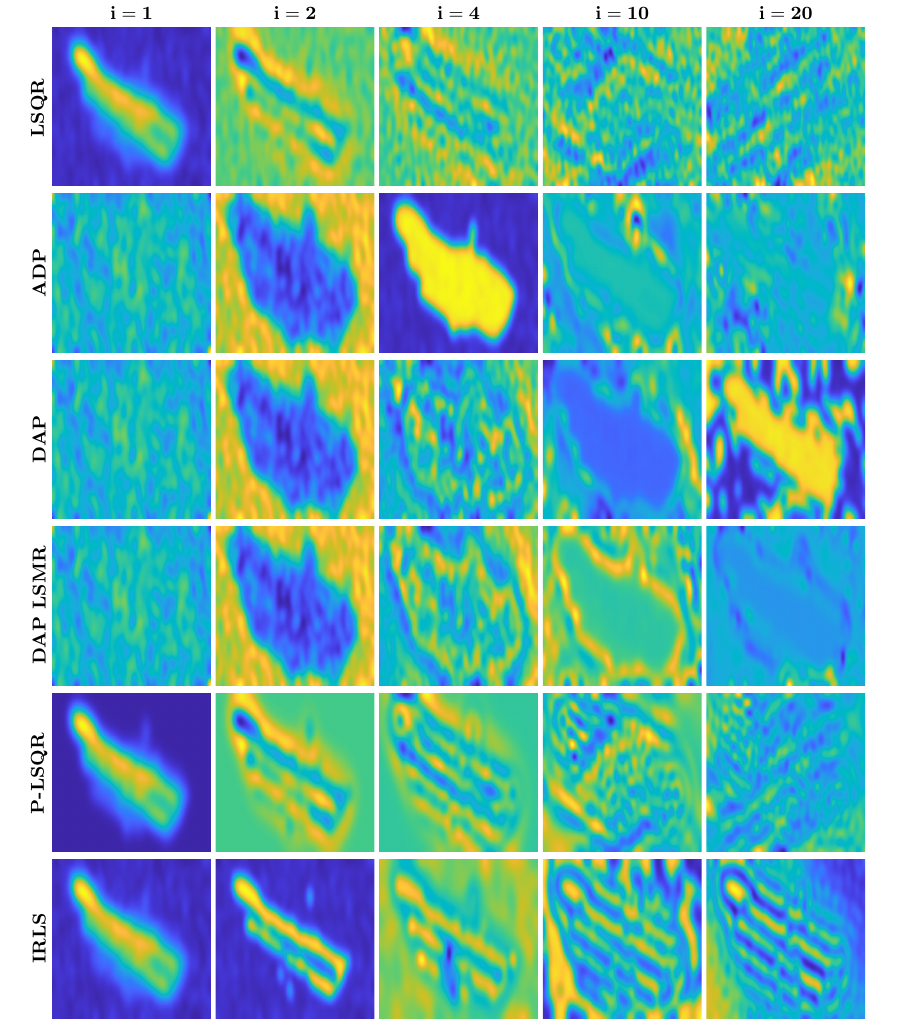} 
    \caption{Example 1. Basis vectors for the solution space, i.e. $\bfV$ columns, for the different compared methods.}
    \label{fig:Ex1_basisvects_extra} 
\end{figure}

\begin{figure}[ht]
\centering
\includegraphics[width=\textwidth]{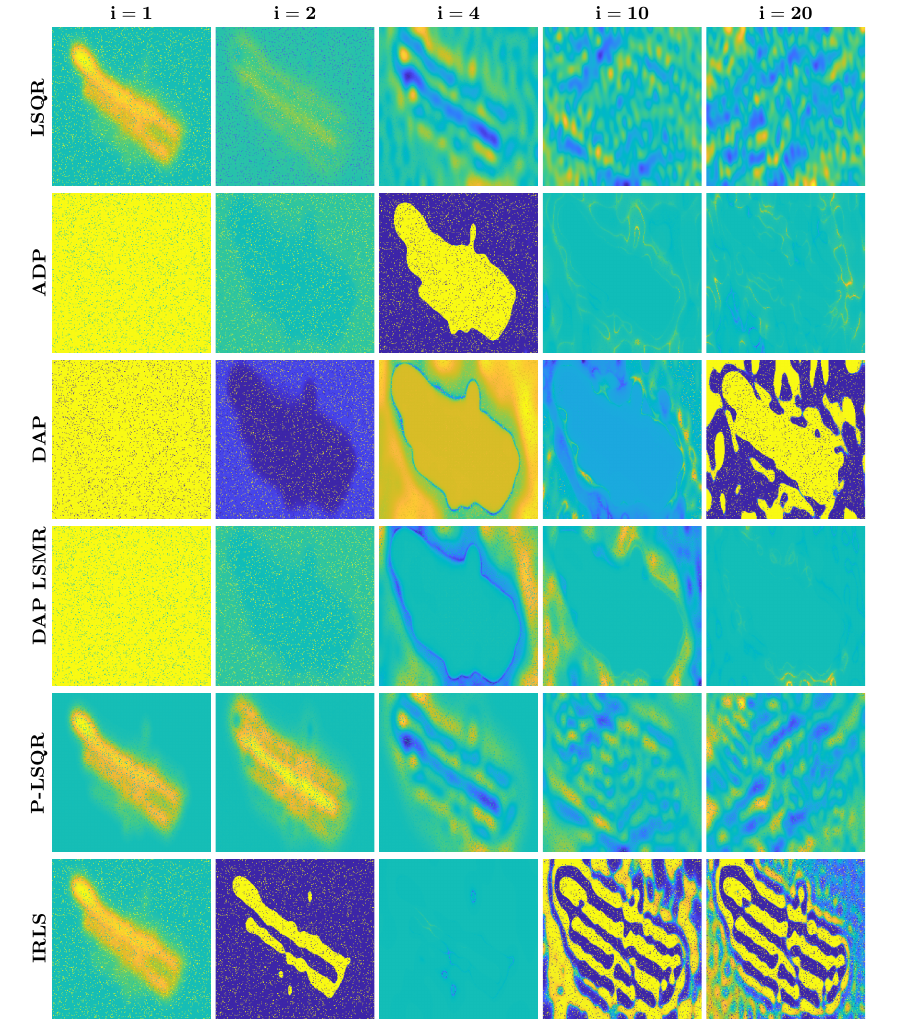} 
    \caption{Example 1. Basis vectors for the residual space, i.e. $\bfZ$ columns, for the different compared methods.}
    \label{fig:Ex1_basisvects_extra_2}
\end{figure}

\begin{figure}[ht]
\centering
\includegraphics[width=\textwidth]{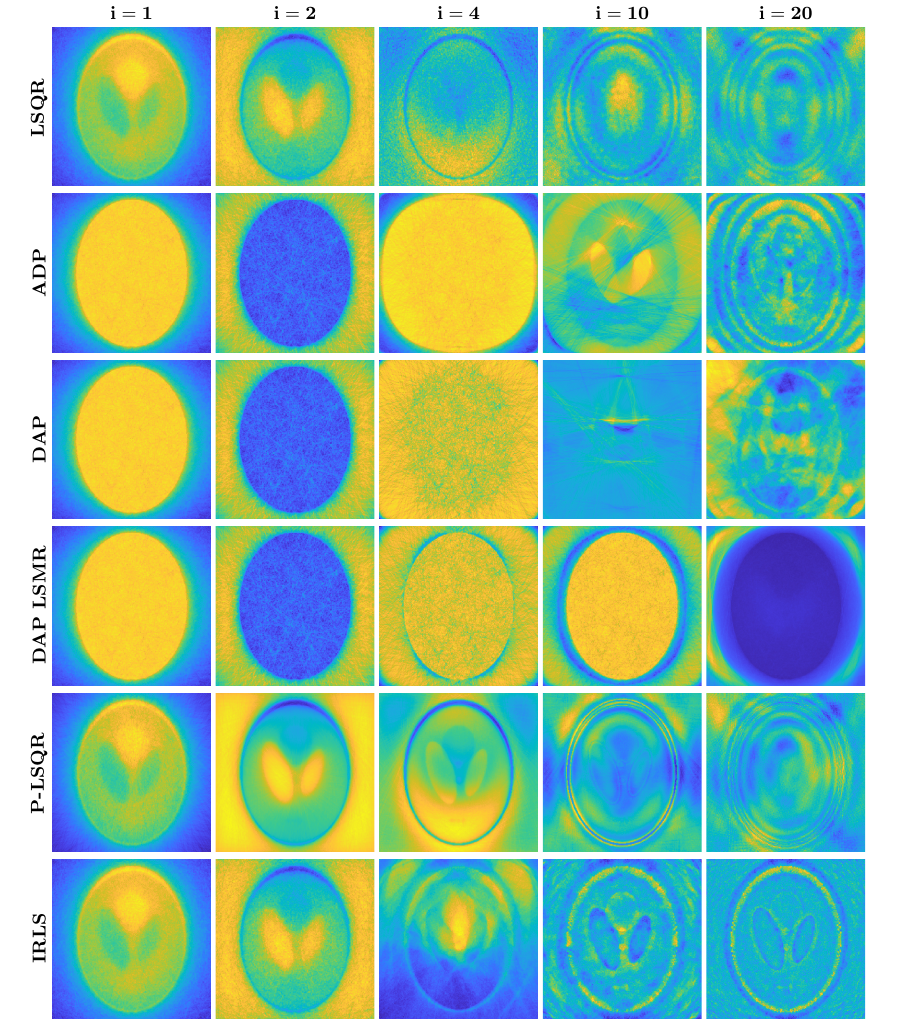}  
    \caption{Example 2. Basis vectors for the solution space, i.e. $\bfV$ columns, for the different compared methods.}
    \label{fig:Ex2_basisvects_extra}
\end{figure}

\begin{figure}[ht]
\centering
\includegraphics[width=0.6\textwidth]{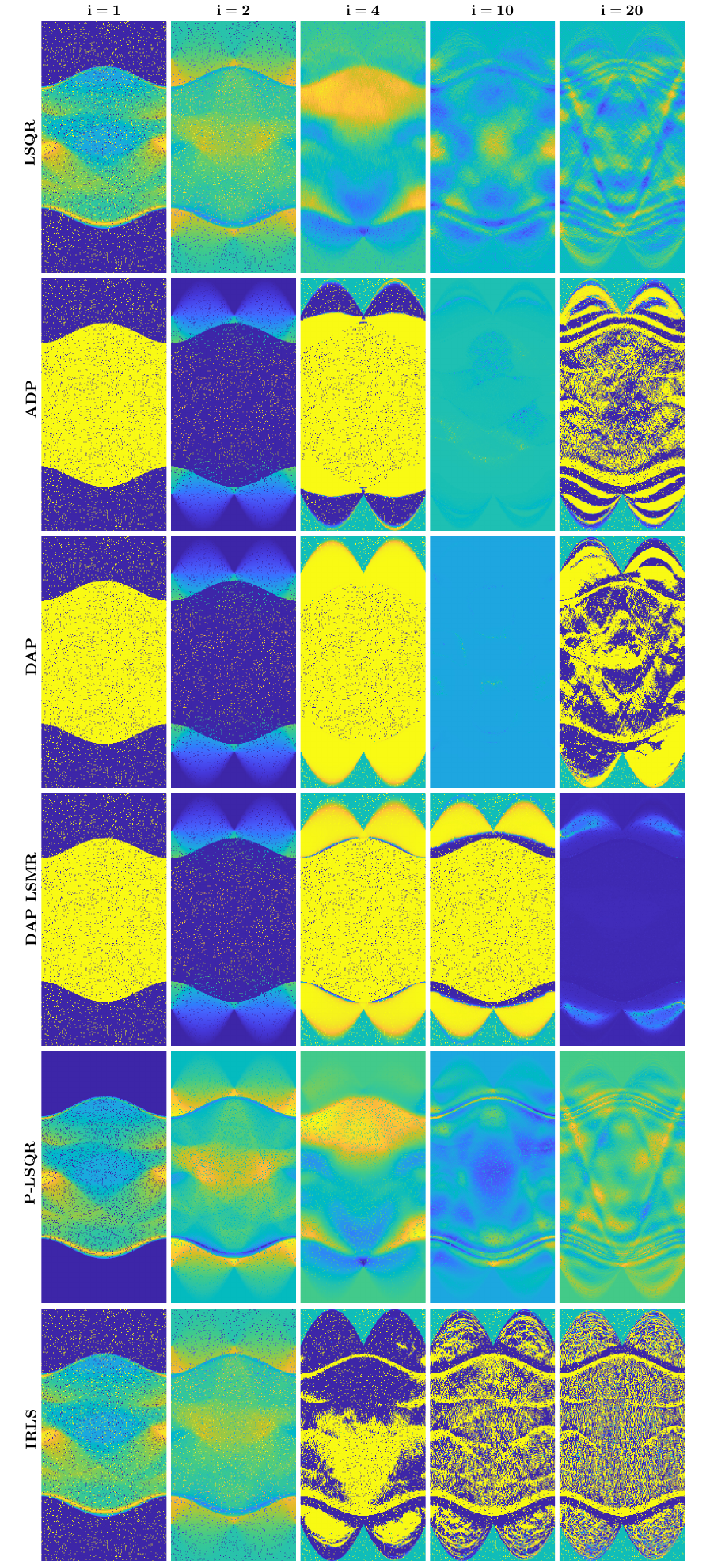}  
    \caption{Example 2. Basis vectors for the residual space, i.e. $\bfZ$ columns, for the different compared methods.}
    \label{fig:Ex2_basisvects_extra_2}
\end{figure}

\end{document}